\documentclass[11pt]{amsart}

\usepackage[latin1]{inputenc}
\usepackage{amsthm}
\usepackage{amsmath}
\usepackage{amssymb}
\usepackage{mathrsfs}
\usepackage{enumitem}
\usepackage{array}
\usepackage[T1]{fontenc}
\usepackage{pifont}

\usepackage{tikz}
\usetikzlibrary{matrix,arrows,decorations.pathmorphing}

\renewcommand{\(}{\left(}
\renewcommand{\)}{\right)}

\newcommand{\sideprime}{\sideset{}{'}}

\usepackage[hidelinks]{hyperref}
\usepackage[margin=1in]{geometry}

\theoremstyle{plain}
\newtheorem{theo}{Theorem}[section]
\newtheorem{prop}[theo]{Proposition}
\newtheorem{lem}[theo]{Lemma}

\theoremstyle{definition}

\newtheorem{defi}{Definition}

\theoremstyle{remark}
\newtheorem{Rk}{Remark}
\newtheorem{ex}{Example}

\DeclareMathOperator{\re}{Re}
\DeclareMathOperator{\dens}{dens}
\DeclareMathOperator{\Var}{Var}

\newcommand*\diff{\mathop{}\!\mathrm{d}}

\begin{document}

\title{Chebyshev's bias for products of irreducible polynomials}
\author{Lucile Devin}
\email{devin@crm.umontreal.ca}
\address{Centre de recherches math\'ematiques,
	Universit\'e de Montr\'eal,
	Pavillon Andr\'e-Aisenstadt,
	2920 Chemin de la tour,
	Montr\'eal, Qu\'ebec, H3T 1J4, Canada}

\author{Xianchang Meng}
\email{xianchang.meng@uni-goettingen.de}
\address{Mathematisches Institut,
Georg-August Universit\"{a}t G\"{o}ttingen,
Bunsenstra{\ss}e 3-5,
D-37073 G\"{o}ttingen,
Germany}

%\date\today

\begin{abstract}
For any $k\geq 1$, this paper studies the number of polynomials having $k$ irreducible factors (counted with or without multiplicities) in $\mathbf{F}_q[t]$ among different arithmetic progressions. 
We obtain asymptotic formulas  for the difference of counting functions uniformly for $k$ in a certain range. In the generic case, the bias dissipates as the degree of the modulus or $k$ gets large, but there are cases when the bias is extreme. In contrast to the case of products of $k$ prime numbers, we show the existence of complete biases in the function field setting, that is the difference function may have constant sign. Several examples illustrate this new phenomenon.

\end{abstract}

\maketitle

\section{Introduction}

\subsection{Background}

The notion of Chebyshev's bias originally refers to the observation in \cite{ChebLetter} that
there seems to be more primes congruent to $3 \bmod 4$ than to $1 \bmod 4$ in initial intervals of the integers.
More generally it is interesting to study the function $ \pi(x; q,a) -\pi(x;q, b)$
where $\pi(x; q,a)$ is the number of primes~$\leq x$ that are congruent to $a \bmod q$.
Under the Generalized Riemann Hypothesis (GRH) and the Linear Independence (LI) conjecture for zeros of the Dirichlet $L$-functions, 
Rubinstein and Sarnak \cite{RS} gave a framework to study Chebyshev's bias quantitatively.
Precisely they showed that the logarithmic density $\delta(q;a,b)$ of the set of $x\geq2$ for which $\pi(x; q,a)>\pi(x;q, b)$ exists and in particular $\delta(4;3,1)\approx 0.9959$.
Many related questions have been asked and answered since then, we refer to the expository articles of Ford and Konyagin \cite{FordKonyagin_expository} and of Granville and Martin \cite{GranvilleMartin} for detailed reviews of the subject.

In this article we consider products of $k$ irreducible polynomials among different congruence classes. Our results are uniform for $k$ in a certain range, and we show that in some cases the \emph{bias} (see Definition~\ref{Defn-bias}) in the distribution can approach any value between $0$ and $1$.
The idea of this paper is motivated by two different generalizations of Chebyshev's bias.

On one hand, Ford and Sneed \cite{FordSneed2010} adapted the observation of Chebyshev's bias to quasi-prime numbers, i.e. numbers with two prime factors $p_1 p_2$ ($p_1=p_2$ included). They showed  under GRH and LI that the direction of the bias for products of two primes is opposite to the bias among primes, and that the bias decreases. Similar results are developed in \cite{DummitGranvilleKisilevsky,Moree2004}.
Recently, under GRH and LI, the second author \cite{Meng2017, Meng2017L} generalized the results of \cite{RS}, \cite{FordSneed2010} and \cite{DummitGranvilleKisilevsky} to products of any $k$ primes among different arithmetic progressions, and observed that the bias changes direction according to the parity  of $k$.

On the other hand, using the analogy between the ring of integers and polynomial rings over finite fields, Cha \cite{Cha2008} adapted the results of \cite{RS} to irreducible polynomials over finite fields. Cha discovered a surprising phenomenon: in the case of polynomial rings there are biases in unexpected directions. Further generalizations have been studied since then in \cite{ChaKim,ChaIm,CFJ,Perret-Gentil}.

Fix a finite field $\mathbf{F}_{q}$ and a polynomial $M \in \mathbf{F}_{q}[T]$ of degree $d \geq 1$,  we study the distribution in congruence classes modulo $M$ of monic polynomials with $k$ irreducible factors.
More precisely let  $A, B\subset (\mathbf{F}_{q}[t]/(M))^{*}$ be subsets of invertible classes modulo $M$, 
for any $k\geq 1$, and any $X \geq 1$
we define the normalized\footnote{Using Sathe--Selberg method, Afshar and Porritt \cite[Th. 2]{AfsharPorritt} gave an asymptotic formula for the number of monic polynomials of degree $X$ with $k$ irreducible factors in congruence classes modulo $M\in \mathbf{F}_{q}[t]$ the main term of which is $\frac{q^{X}(\log X)^{k-1}}{X (k-1)! \phi(M)}$ in the case $k = o(\log X)$ and the modulus $M$ does not vary with $X$. In this paper, we focus on the error terms and expect to have square-root cancellation in the error terms.} difference function
\begin{multline*}
\Delta_{f_k}(X; M, A, B) \\ :=  \frac{X (k-1)!}{q^{X/2}(\log X)^{k-1}} \Big(\frac{1}{\lvert A\rvert}\lvert \lbrace N \in \mathbf{F}_{q}[t]:  N \text{ monic, } \deg{N} \leq X,~ f(N)=k, N \bmod M \in A \rbrace\rvert \\
   - \frac{1}{\lvert B\rvert}\lvert \lbrace N \in \mathbf{F}_{q}[t]:  N \text{ monic, } \deg{N} \leq X,~ f(N)=k, N \bmod M \in B \rbrace\rvert \Big) \\
\end{multline*}
where $f=\Omega$ or $\omega$ is the number of prime factors counted with or without multiplicities. 
We study the distribution of the values of the function $\Delta_{f_k}(X; M, A, B)$, in particular we are interested in the bias of this function towards positive values.
\begin{defi}[\emph{Bias}]\label{Defn-bias}
Let $F:\mathbf{N}\rightarrow\mathbf{R}$ be a real function, we define the \emph{bias} of $F$ towards positive values as the natural density (if it exists) of the set of integers having positive image by $F$:
\begin{align*}
\dens(F >0) = \lim_{X\rightarrow\infty}\frac{ \lvert\lbrace n \leq X : F(n)>0 \rbrace\rvert  }{X}.
\end{align*}
If the limit does not exist, we say that the bias is not well defined.
\end{defi}

\subsection{Values of the bias}
In this section we present our main result which is the consequence of the asymptotic formula obtained in Theorem~\ref{Th_Difference_k_general_deg<X}.

Given a field~$\mathbf{F}_{q}$ with odd characteristic, and a square-free polynomial  $M$  in $\mathbf{F}_{q}$, we examine more carefully the case of races between quadratic residues ($A = \square$) and non-quadratic residues ($B =\boxtimes$) modulo $M$.

We say that $M$ satisfies (LI\ding{70}) if the multi-set $$\lbrace\pi\rbrace\cup\bigcup\limits_{\substack{\chi \bmod M \\ \chi^{2} = \chi_0 , \chi \neq \chi_0}}\lbrace \gamma \in [0,\pi] :  L(\frac{1}{2} + i\gamma, \chi) = 0  \rbrace$$ is linearly independent over $\mathbf{Q}$ (see Section~\ref{subsec_Lfunctions} for the definition of the $L$-functions).

We study the variation of the values of the bias when the degree of the modulus $M$ gets large.
 In particular, we show that the values of the bias are dense in $[\tfrac{1}{2},1]$.

\begin{theo}\label{Th central limit for M under LI}
Let $\mathbf{F}_{q}$ be a finite field of odd characteristic and $k$ a positive integer.
Suppose that for every $d,r$ large enough, there exist a monic polynomial $M_{d,r} \in \mathbf{F}_q[t]$ with 
\begin{enumerate}
    \item $\deg M_{d,r} =d$,
    \item  $\omega(M_{d,r}) = r$,
    \item $M_{d,r}$ satisfies \emph{(LI\ding{70})}.
\end{enumerate}
Then for $f=\Omega$ or $\omega$,
$$ \overline{\lbrace \dens ((\epsilon_f)^k \Delta_{f_k}(\cdot;M_{d,r},\square,\boxtimes) > 0) : d \geq1, r\geq1  \rbrace} = [\tfrac12,1], $$
 where $\epsilon_{\Omega} = -1$ and $\epsilon_{\omega} = 1$.
\end{theo}

\begin{Rk}
Note that, when $k$ is odd, we obtain that the possible values of $\dens ( \Delta_{\Omega_k}(\cdot;M,\square,\boxtimes) > 0)$ are dense in $[0,\tfrac{1}{2}]$, when $M$ varies in $\mathbf{F}_q[t]$, while the function $\Delta_{\omega_k}(\cdot;M,\square,\boxtimes) $ is biased in the direction of quadratic residues independently of the parity of $k$.
\end{Rk}

From \cite[Prop.~1.1]{Kowalski2010}, we expect the hypothesis (LI\ding{70}) to be true for most of the monic square-free polynomials~$M \in \mathbf{F}_{q}[t]$ when $q$ is large enough. When $d, r$ are large, the set of polynomials of degree $d$ having $r$ irreducible factors should be big enough to contain at least one polynomial satisfying (LI\ding{70}).
However, in Proposition~\ref{Prop Chebyshev many factors}, similarly to \cite[Th.~1.2]{Fiorilli_HighlyBiased}, we only need an hypothesis on the multiplicity of the zeros to prove the existence of extreme biases.

In \cite[Th.~6.2]{Cha2008}, Cha considered the case $k=1$ and showed that the values of the bias $ \dens (\Delta_{\Omega_1}(\cdot;M_{d,r},\square,\boxtimes) > 0)$ approach $\tfrac{1}{2}$ when $M$ varies among irreducible polynomials of increasing degree.
In the case $k=1$, Fiorilli \cite[Th.~1.1]{Fiorilli_HighlyBiased} proved that the values of the biases in prime number races between non-quadratic and quadratic residues are dense in $[\tfrac{1}{2},1]$.
We also show in Proposition~\ref{Prop using Honda Tate} that the values $\tfrac12$ and $1$ can be obtained as values of $\dens ((-1)^k \Delta_{f_k}(\cdot;M,\square,\boxtimes) > 0)$, when $q>3$. These values are obtained for polynomials $M$ not satisfying (LI\ding{70}).

In the case $q =5$, Cha showed that there exists $M \in \mathbf{F}_5[t]$ with $\dens ( \Delta_{\Omega_{1}}(\cdot;M,\square,\boxtimes) > 0) = 0.6$, uncovering a bias in ``the wrong direction'', we wonder if such a phenomenon occurs for any $q$ and $k$.

\subsection{Asymptotic formulas}\label{subsec_Lfunctions}

Before stating the asymptotic formulas, let us set some notations.
For $M \in \mathbf{F}_{q}[t]$ we denote $\phi(M) = \lvert \left( \mathbf{F}_{q}[t]/ (M) \right)^* \rvert$ the number of invertible congruence classes modulo~$M$.

Recall that we define the Dirichlet $L$-function associated to a Dirichlet character $\chi$ by 
$$L(s, \chi) = \sum_{a \text{ monic }} \frac{\chi(a)}{\lvert a \rvert^{s}}$$
where $\lvert a \rvert = q^{\deg a}$.
It can also be written as an Euler product over the irreducible polynomials:
$$L(s, \chi) = \prod_{P \text{ irreducible }}\left( 1- \frac{\chi(P)}{\lvert P \rvert^s}\right)^{-1}.$$

Recall that (e.g. \cite[Prop.~4.3]{Rosen2002}), for $\chi \neq \chi_0$ a Dirichlet character modulo~$M \in\mathbf{F}_{q}[t]$, the Dirichlet~$L$-function  $L(\chi,s)$ is a polynomial in $q^{-s}=u$ of degree at most $\deg M-1$. Thanks to the deep work of Weil \cite{Weil_RH}, we know that the analogue of the Riemann Hypothesis is satisfied.

In the following we denote $\alpha_{j}(\chi)= \sqrt{q}e^{i\gamma_{j}(\chi)}$, $\gamma_{j}(\chi) \in (-\pi, \pi)\setminus \lbrace 0 \rbrace$, the distinct non-real inverse zeros of $\mathcal{L}(u,\chi) = \mathcal{L}(q^{-s},\chi) = L(s,\chi)$ of norm $\sqrt{q}$, with multiplicity $m_j(\chi)$.
The real inverse zeros will play an important role; we denote $m_{\pm}(\chi)$ the multiplicity of $\pm \sqrt{q}$ as an inverse zero of $\mathcal{L}(u,\chi)$, and $d_{\chi}$ the number of distinct non-real inverse zeros or norm $\sqrt{q}$. 
We summarize the notations in the following formula: 
\begin{equation}\label{Not_L-function}
\mathcal{L}(u,\chi) = (1-u \sqrt{q})^{m_+(\chi)}(1+u \sqrt{q})^{m_-(\chi)} \prod_{j=1}^{d_{\chi}} (1 - u \alpha_{j}(\chi) )^{m_{j}(\chi)}\prod_{j'=1}^{d'_{\chi}} (1- u\beta_{j'}(\chi))
\end{equation}
where $\lvert \beta_{j'}(\chi) \rvert =1$.
Recently Wanlin Li proved \cite[Th.~1.2]{Li2018} that $m_{+}(\chi) >0$ for some primitive quadratic character $\chi$ over $\mathbf{F}_{q}[t]$ for any odd $q$. This result disproves the analogue of a conjecture of Chowla about the existence of central zeros.
We present some of such examples in Section~\ref{subsec_examples_realZero} to exhibit large biases. As $k$ increases, we observe a new phenomenon: such characters can induce complete biases in races between quadratic and non-quadratic residues (see Section \ref{subsec_examples_realZero}), those biases do not dissipate as $k$ gets large (see Proposition~\ref{Prop k limit}). 

Denote
$$\gamma(M) = \min\limits_{\chi \bmod M}\min\limits_{1\leq i\neq j\leq d_{\chi}}(\lbrace \lvert \gamma_i(\chi) - \gamma_j(\chi) \rvert, \lvert \gamma_{i}(\chi) \rvert, \lvert \pi -\gamma_{i}(\chi) \rvert \rbrace).$$

We have the following asymptotic formulas without any conditions uniformly for $k$ in some reasonable range (for example,  $k\leq \frac{0.99\log\log X}{\log d}$).
\begin{theo}\label{Th_Difference_k_general_deg<X}
Let $M \in \mathbf{F}_{q}[t]$ be a non-constant polynomial of degree $d$, 
and let  $A, B \subset (\mathbf{F}_{q}[t]/(M))^*$ be two sets of invertible classes modulo $M$. 
For any integer $k \geq1 $, 
with $k=o((\log X)^{\frac12})$, one has
    \begin{multline}\label{Form_deg=n}
   \Delta_{\Omega_k}(X; M, A,B)  \\
    = (-1)^k \Bigg\{ \sum_{\chi \bmod M}c(\chi,A,B)\bigg( \(m_+(\chi) +\tfrac{\delta(\chi^2)}{2}\)^k \frac{\sqrt{q}}{\sqrt{q}-1} +\(m_-(\chi)+\tfrac{\delta(\chi^2)}{2}\)^k \frac{\sqrt{q}}{\sqrt{q} + 1}(-1)^{X} \\
    +\sum_{\gamma_j(\chi)\neq 0, \pi} m_j(\chi)^k \frac{\alpha_j(\chi)}{\alpha_j(\chi) -1} e^{iX\gamma_{j}(\chi)} \bigg)  + O\bigg(\frac{d^k k(k-1)}{\gamma(M)\log X} + dq^{-X/6}\bigg) \Bigg\},
	\end{multline}
	and if $q\geq 5$,
	   \begin{multline}\label{Form_deg=n_littleomega}
   \Delta_{\omega_k}(X; M, A,B)  \\
    = (-1)^k \Bigg\{ \sum_{\chi \bmod M}c(\chi,A,B)\bigg( \(m_+(\chi) -\tfrac{\delta(\chi^2)}{2}\)^k \frac{\sqrt{q}}{\sqrt{q}-1} +\(m_-(\chi)-\tfrac{\delta(\chi^2)}{2}\)^k \frac{\sqrt{q}}{\sqrt{q} + 1}(-1)^{X} \\
    +\sum_{\gamma_j(\chi)\neq 0, \pi} m_j(\chi)^k \frac{\alpha_j(\chi)}{\alpha_j(\chi) -1} e^{iX\gamma_{j}(\chi)} \bigg)  + O\bigg(\frac{d^k k(k-1)}{\gamma(M)\log X} + dq^{-X/6}\bigg) \Bigg\},
	\end{multline}
	where the implicit constants are absolute, $\delta(\chi^2) = 1$ if $\chi$ is real and $0$ otherwise, 
	and 
    $$c(\chi,A,B) = 
    \frac{1}{\phi(M)}\bigg( \frac{1}{\lvert A\rvert}\sum_{a\in A}\bar\chi(a) - 
    \frac{1}{\lvert B\rvert}\sum_{b\in B}\bar\chi(b) \bigg).$$
 \end{theo}
    
     This theorem follows from the asymptotic formula obtained in Section~\ref{Sec_proof_deg<X}.
     The method we use here is not a straightforward generalization of the method used in \cite{Cha2008} since the analogue of the weighted form of counting function is not ready to detect products of irreducible elements (see \cite{Meng2017}). Different from the results in \cite{RS}, \cite{FordSneed2010} and \cite{Meng2017}, we obtain asymptotic formulas for the corresponding difference functions unconditionally, and the density we derive in this paper is the natural density rather than the logarithmic density. Our starting point is motivated by a combinatorial idea in \cite{Meng2017}, but the main proof is not a parallel translation since the desired counting function is not derived as Meng did in \cite{Meng2017} using Perron's formula. 

\begin{Rk}
\label{Rk_on_Th_deg<n}
\begin{enumerate}[label = \roman*)]
 
 \item\label{Item_IntegerCase} Note that in the case $\Delta_{\Omega_1}$ this result is \cite[Th.~2.5]{Cha2008}. Our formulas are more general analogues of the result in \cite{Meng2017} including the multiplicities of the zeros.
 
 \item\label{Item_CompareOmegas} As $\lvert m_+(\chi) -\frac{\delta(\chi^2)}{2} \rvert \leq \lvert m_+(\chi) +\frac{\delta(\chi^2)}{2} \rvert  $, we expect a more important bias in the race between polynomials with $\Omega(N) = k$ than in the race between polynomials with $\omega(N) = k$. Note also that if $m_{+}(\chi) =0$ for all $\chi$, the two mean values have different sign when $k$ is odd, hence we expect the two biases to be in different directions when $k$ is odd. 
 
\item\label{Item_largestMult} We observe from the formula that the inverse zeros with largest multiplicity will determine the behavior of the function as $k$ grows. This is the point of Proposition~\ref{Prop k limit} below. Moreover the real zeros play an important role in determining the bias.

\item\label{Item_Range_k} For degree $X$ polynomials, the typical number of irreducible factors is $\log X$. Hence, one may expect an asymptotic formula which holds for $k\ll \log X$, or at least for $k=o(\log X)$. However, we are not able to reach this range in general, and the factor $d^k$ in the error term is inevitable in our proof. Through personal communication, we know that Sam Porritt  is currently using a different method to study these asymptotic formulas \cite{Porritt}.

\end{enumerate}
\end{Rk}

 In the case of the race of quadratic residues against non-quadratic residues modulo $M$, the expressions in \eqref{Form_deg=n} and \eqref{Form_deg=n_littleomega} can be simplified. 
 This is studied in more detail in Section~\ref{Sec_Examples}.
For the race between polynomials with $\Omega(N) = k$, we expect a bias in the direction of quadratic residues or non-quadratic residues according to the parity of $k$.
We show that the existence of the real zero  $\sqrt{q}$ sometimes leads to extreme biases.

In the generic case, we expect $m_{\pm} = 0$ and that the other zeros are simple,
then the asymptotic formulas in Theorem~\ref{Th_Difference_k_general_deg<X} give a connection between $\Delta_{f_k}(X; M, A,B)$ and $\Delta_{f_1}(X; M, A,B)$ ($f=\Omega$ or $\omega$), similar to the case of products of primes \cite[Cor.~1.1, Cor.~2.1]{Meng2017}. 
We expect in this case that the polynomials
with $\Omega(N) = k$ have preference for quadratic non-residue classes when $k$ is odd; and when $k$ is even, such polynomials
have preference for quadratic residue classes. However, the polynomials with $\omega(N) = k$ always have preference for quadratic residue
classes. Moreover, as $k$ increases, the biases become smaller and smaller for both of the two
cases. This observation is justified by Proposition~\ref{Prop k limit}(Expected generic case).

\subsection{Further behaviour of the bias}

The asymptotic formula from Theorem~\ref{Th_Difference_k_general_deg<X} helps the understanding of the bias in the distribution of polynomials with a certain number of irreducible factors in congruence classes.

 In the case of a sequence of polynomials with few irreducible factors, we give a precise rate of convergence of the bias to $\tfrac{1}{2}$ in Theorem~\ref{Th central limit for M under LI} in the following result.

\begin{theo}\label{Th_CentralLimit}
Let $\lbrace M \rbrace$ be a sequence of square-free polynomials in $\mathbf{F}_{q}[t]$ satisfying \emph{(LI\ding{70})} and such that $\frac{2^{\omega(M)}}{\deg M} \rightarrow 0$.
Then, for $f = \Omega$ or $\omega$, as $\deg M \rightarrow\infty$,
the limiting distribution $\mu_{M,f_k}^{\mathrm{norm}}$ of 
\begin{align*}
    &\Delta_{f_k}^{\mathrm{norm}}(X;M) := \frac{\lvert \boxtimes\rvert \sqrt{q-1} }{\sqrt{q2^{\omega(M)-1}\deg M}}\Delta_{f_k}(X;M,\square,\boxtimes)
\end{align*}
exists and  converges weakly to the standard Gaussian distribution.
More precisely, one has
\begin{equation*}
    \sup_{x\in\mathbf{R}}\left\lvert \int_{-\infty}^{x}\diff\mu_{M,f_k}^{\mathrm{norm}} - \frac{1}{\sqrt{2\pi}}\int_{-\infty}^{x} e^{-t^2/2} \diff t \right\rvert \ll  \frac{\sqrt{2^{\omega(M)}}}{2^k\sqrt{\deg M}} + \frac{\log \deg M}{\deg M}.
\end{equation*}
In particular the bias dissipates as $\deg M$ gets large.
\end{theo}

\begin{Rk}
Note that a sequence of irreducible polynomials $M$ with increasing degree satisfies the hypothesis $\frac{2^{\omega(M)}}{\deg M} \rightarrow 0$, thus Theorem~\ref{Th_CentralLimit} generalizes \cite[Th.~6.2]{Cha2008}.
We observe in particular that the rate of convergence to the Gaussian distribution increases with $k$, this justifies an observation in the number fields setting \cite{Meng2017}: the race seems to be less biased when $k$ is large.
\end{Rk}

In the other direction, fixing a modulus and letting $k$ grow, we obtain the following result.
\begin{theo}
\label{Th central limit for k under LI}
Let $\mathbf{F}_{q}$ be a finite field of odd characteristic and  $M \in \mathbf{F}_q[t]$ satisfying \emph{(LI\ding{70})}.
Then, for $f =\Omega$ or $\omega$, the bias in the distributions of $\Delta_{f_k}(X;M,\square,\boxtimes)$ dissipates as $k\rightarrow\infty$.
\end{theo}
This is a corollary of Proposition~\ref{Prop k limit} which is more general and unconditional.

\section{Limiting distribution and bias}\label{Sec_LimDist}

In this section, the assertions are given in the context of almost periodic functions as in \cite{ANS}, as we expect these to be useful for other work on Chebyshev's bias over function fields.
Our main results are based on the existence of a limiting distribution for functions defined over the integers, let us briefly recall the definitions and ideas to obtain such results.

\begin{defi}
Let $F:\mathbf{N}\rightarrow\mathbf{R}$ be a real function, 
we say that $F$ admits a limiting distribution if there exists a probability measure $\mu$ on Borel sets in $\mathbf{R}$ such that
for any bounded Lipschitz continuous function $g$, we have
\begin{align*}
\lim_{Y\rightarrow\infty}\frac{1}{Y}\sum_{n\leq Y}g(F(n)) = 
\int_{\mathbf{R}}g(t)\diff\mu(t).
\end{align*}
We call $\mu$ the limiting distribution of the function $F$.
\end{defi}

\begin{Rk}
Note that if the function $F$ admits a limiting distribution $\mu$, and that $\mu(\lbrace 0\rbrace) = 0$, then by dominated convergence theorem,  the bias of $F$ towards positive values (see Definition~\ref{Defn-bias}) is well defined and we have $\dens(F >0) = \mu((0,\infty))$. 
\end{Rk}

We focus on the limiting distribution to study the bias of the difference function.
For $f = \Omega$ or $\omega$, and for any $k\geq 1$, the fact that the function $\Delta_{f_k}(\cdot; M, A,B)$ admits a limiting distribution follows directly from the asymptotic formula of Theorem~\ref{Th_Difference_k_general_deg<X} and the following result.

\begin{prop}\label{Prop_limitingDist}
Let $\gamma_2,\ldots,\gamma_N \in (0,\pi)$ be distinct real numbers.
For any $C_0,c_1\in \mathbf{R}$, $c_2,\ldots, c_N \in \mathbf{C}^{*}$,
let $F:\mathbf{N} \rightarrow \mathbf{R}$ be a function satisfying
\begin{equation}\label{Almost periodic function}
F(n) = C_0 + c_1e^{in\pi} + \sum_{j=2}^{N}\( c_je^{in\gamma_j} + \overline{c_j}e^{-in\gamma_j}\) + o(1)
\end{equation}
as $n\rightarrow \infty$.
Then the function $F$ admits a limiting distribution $\mu$ with mean value $C_0$ and variance $c_1^2 + 2\sum\limits_{j=1}^{N}\lvert c_j\rvert^2$.
Moreover
\begin{enumerate}[label = \roman*)]
\item\label{Item_Support} the measure $\mu$ has support in $\left[C_0 - \lvert c_1\rvert - \sum_{j=2}^{N}2\lvert c_j\rvert, C_0 + \lvert c_1\rvert +\sum_{j=2}^{N}2\lvert c_j\rvert\right]$,\\
in particular, if $\lvert C_0 \rvert > \lvert c_1\rvert +  \sum_{j=2}^{N}2\lvert c_j\rvert $ then $\dens(C_0F >0) = 1$;
\item\label{Item_continuous} if there exists $j\in \lbrace 2, \ldots, N\rbrace$ such that $\gamma_j \notin \mathbf{Q}\pi$, then $\mu$ is continuous,\\ 
in particular $\dens(F >0) = \mu((0,\infty))$;
\item\label{Item_symmetry} if the smallest sub-torus of $\mathbf{T}^{N}$ containing $\lbrace (n\pi, n\gamma_2,\ldots, n\gamma_N) : n\in \mathbf{Z} \rbrace$ is symmetric,
then the distribution $\mu$ is symmetric with respect to $C_0$;
\item\label{Item_Fourier} if the set $\lbrace \pi, \gamma_2,\ldots,\gamma_N\rbrace$ is linearly independent over $\mathbf{Q}$, then the Fourier transform $\hat\mu$ of the measure $\mu$ is given by
\[
    \hat\mu(\xi) = e^{-iC_{0}\xi}\cos(c_1\xi)
    \prod_{j=2}^{N}J_{0}\left( 2\lvert c_j \rvert \xi \right),
\]
where $J_{0}(z) = \int_{-\pi}^{\pi}\exp\left(iz\cos(\theta)\right) \frac{\diff\theta}{2\pi}$ is the $0$-th Bessel function of the first kind.
\end{enumerate}
\end{prop}

Kowalski, \cite[Prop.~1.1]{Kowalski2010}, showed that in certain families of polynomials $M \in \mathbf{F}_q[t]$, the hypothesis of Linear Independence (LI) is satisfied generically when $q$ is large (with fixed characteristic) for the $L$-function of the primitive quadratic character modulo $M$.
 That is, the imaginary parts of the zeros of $L(\cdot,\chi_{M})$ are linearly independent over $\mathbf{Q}$. In particular the hypotheses in \ref{Item_continuous}, \ref{Item_symmetry} and \ref{Item_Fourier} are satisfied generically for $F = \pi_{k}(\cdot,\chi_M)$ (see \eqref{Eq defi pi(chi)}).
 We expect this to hold more generally for example when racing between quadratic residue and non-quadratic residues as in Section~\ref{Sec_Examples}.
The Linear Independence has also been proved generically in other families of $L$-functions over functions fields \cite{CFJ_Indep,Perret-Gentil}. 
Proposition~\ref{Prop_limitingDist} is a consequence of a general version of the Kronecker--Weyl Equidistribution Theorem (see \cite[Lem.~2.7]{Humphries}, \cite[Th.~4.2]{Devin2018}, also \cite[Lem.~B.3]{MartinNg}).
\begin{lem}[Kronecker--Weyl]\label{Th_KW}
	Let $\gamma_1,\ldots,\gamma_N \in \mathbf{R}$ be real numbers.
	Denote $A(\gamma)$ the closure of the $1$-parameter group $\lbrace y(\gamma_{1},\ldots,\gamma_{N}) : y\in\mathbf{Z}\rbrace/(2\pi\mathbf{Z})^{N}$ in the $N$-dimensional torus $\mathbf{T}^{N}:= (\mathbf{R}/2\pi\mathbf{Z})^{N}$.
    	Then $A(\gamma)$ is a sub-torus of $\mathbf{T}^{N}$ and we have
	for any continuous function $h: \mathbf{T}^{N}\rightarrow \mathbf{C}$,
	\begin{equation*}
\lim_{Y\rightarrow\infty}\frac{1}{Y}\sum_{n=0}^{Y}h(n\gamma_{1},\ldots,n\gamma_{N})
	= \int_{A(\gamma)}h(a)\diff\omega_{A(\gamma)}(a)
	\end{equation*}
	where $\omega_{A(\gamma)}$ is the normalized Haar measure on $A(\gamma)$. 
\end{lem}

\begin{proof}[Proof of Proposition~\ref{Prop_limitingDist}]
As in the proof of \cite[Th.~3.2]{Cha2008}, we associate Lemma~\ref{Th_KW} with the asymptotic formula of Proposition~\ref{Prop_limitingDist} and Helly's selection theorem \cite[Th.~25.8 and Th.~25.10]{Billingsley}. From this, one can show that the corresponding limiting distribution exists and is a push-forward of the Haar measure on the sub-torus generated by the the $\gamma_j$'s.

Then \ref{Item_Support} is straightforward, and since the measure has compact support, its moments can be computed using compactly supported approximations of polynomials, this gives the result on the mean value and variance.
The point \ref{Item_continuous} follows from the same lines as \cite[Th.~2.2]{Devin2018} using the fact that the set of zeros is finite and being more careful about the rational multiples of $\pi$ to ensure that the sub-torus is not discrete (see also \cite[Th.~4]{Devin2019}). 
The point \ref{Item_symmetry} follows directly from the proof of \cite[Th.~2.3]{Devin2018}.

To prove the point \ref{Item_Fourier}, we compute the Fourier transform:
\begin{align*}
    \hat\mu(\xi) &= \lim_{Y\rightarrow\infty} \frac{1}{Y} \sum_{n\leq Y}\exp(-i\xi F(n)) \\
    &= e^{-iC_0 \xi} \int_{A(\pi,\gamma_1,\ldots,\gamma_N)}
    \exp\Bigg(-i\xi\bigg( c_1(-1)^{a_1} + \sum_{j=2}^{N}2\re(c_je^{ia_j})  \bigg)\Bigg) \diff\omega(a) \\
    &= e^{-iC_0\xi} \frac{1}{2}\left( e^{i\xi c_1} + e^{-i\xi c_1} \right) \prod_{j=2}^{N}\int_{-\pi}^{\pi}\exp\left(i\xi 2 \lvert c_j\rvert \cos(\theta)\right) \frac{\diff\theta}{2\pi},
\end{align*}
where in the last line we use the linear independence to write $A(\pi,\gamma_1,\ldots,\gamma_N) = \lbrace 0, \pi\rbrace \times \mathbf{T}^{N-1}$, and the corresponding Haar measure as the product of the Haar measures. This concludes the proof.
\end{proof}

\begin{Rk}
Note that in the case all the $\gamma_j$'s are rational multiples of $\pi$, then the main term in the asymptotic expansion \eqref{Almost periodic function} is a periodic function. Thus the limiting distribution obtained in Proposition~\ref{Prop_limitingDist} is a linear combination of Dirac deltas supported on the image of this periodic function. If this image does not contain $0$, the limiting distribution has no mass at the point $0$ hence the bias is well defined. 
Otherwise the determination of the bias requires to study lower order terms in the asymptotic expansion, which are for now out of reach.
\end{Rk}

\section{Special values of the bias}\label{Sec_Examples}
In this section, we assume that the field $\mathbf{F}_{q}$ has characteristic $\neq 2$ and that the polynomial $M$ is square-free. When $q$ and the degree of $M$ are small, it is possible to compute the Dirichlet $L$-functions associated to the quadratic characters modulo $M$ explicitly. 
In particular, we illustrate our results in the case of races between quadratic residues ($\square$) and non-quadratic residues ($\boxtimes$) modulo~$M$.
In this case the asymptotic formula of Theorem~\ref{Th_Difference_k_general_deg<X} is a sum over quadratic characters.
Indeed, let $\chi$ be a non-trivial, non-quadratic character, it induces a non-trivial character on the subgroup $\square$ of quadratic residues, by orthogonality one has
$c(\chi,\square,\boxtimes) = 0.$
For $\chi$ a quadratic character, one has
\begin{align*}
    c(\chi,\square,\boxtimes) &= 
    \frac{1}{\phi(M)}\( \frac{1}{\lvert \square\rvert}\sum_{a\in \square}1 - 
    \frac{-1}{\lvert \boxtimes\rvert} \sum_{a\in \square}1 \)\\
    &= \frac{1}{\phi(M)}\( 1 +
    \frac{\lvert \square\rvert}{\lvert \boxtimes\rvert} \) = 
     \frac{1}{\lvert \boxtimes\rvert}.
\end{align*}
Thus, for $ k =o((\log X)^{\frac12})$, one has
    \begin{multline}\label{Formula_Res_vs_NonRes}
    \Delta_{\Omega_k}(X; M, \square,\boxtimes)  \\
    = \frac{(-1)^k}{\lvert \boxtimes\rvert} \Bigg\{ \sum_{\substack{\chi \bmod M \\ \chi^2 = \chi_0 \\ \chi\neq \chi_0}}\Bigg(\ \(m_+(\chi) +\tfrac{1}{2}\)^k \frac{\sqrt{q}}{\sqrt{q}-1} +\(m_-(\chi)+\tfrac{1}{2}\)^k \frac{\sqrt{q}}{\sqrt{q} + 1}(-1)^{X} \\
    +\sum_{\gamma_j\neq 0, \pi} m_j(\chi)^k \frac{\alpha_j(\chi)}{\alpha_j(\chi) -1} e^{iX\gamma_{j}(\chi)} \Bigg)  + O\(\frac{d^k k^2}{\gamma(M)\log X}\) \Bigg\},
	\end{multline}
	and, if $q \geq 5$, 
	    \begin{multline}\label{Formula_Res_vs_NonRes_littleomega}
   \Delta_{\omega_k}(X; M, \square,\boxtimes)  \\
    = \frac{(-1)^k}{\lvert \boxtimes\rvert} \Bigg\{ \sum_{\substack{\chi \bmod M \\ \chi^2 = \chi_0 \\ \chi\neq \chi_0}}\Bigg(\ \(m_+(\chi) -\tfrac{1}{2}\)^k \frac{\sqrt{q}}{\sqrt{q}-1} +\(m_-(\chi)-\tfrac{1}{2}\)^k \frac{\sqrt{q}}{\sqrt{q} + 1}(-1)^{X} \\
    +\sum_{\gamma_j\neq 0, \pi} m_j(\chi)^k \frac{\alpha_j(\chi)}{\alpha_j(\chi) -1} e^{iX\gamma_{j}(\chi)} \Bigg)  + O\(\frac{d^k k^2}{\gamma(M)\log X}\) \Bigg\}.
	\end{multline}
By Proposition~\ref{Prop_limitingDist}, we know that for all $k$ the function in \eqref{Formula_Res_vs_NonRes} admits a limiting distribution $\mu_{M,\Omega_k}$ with  mean value
\begin{equation}\label{Form_MeanValue}
\mathbf{E}\mu_{M,\Omega_k} = \frac{(-1)^k}{\lvert \boxtimes\rvert} \sum_{\substack{\chi \bmod M \\ \chi^2 = \chi_0 \\ \chi\neq \chi_0}}\ \(m_+(\chi) +\frac{1}{2}\)^k \frac{\sqrt{q}}{\sqrt{q}-1},  
\end{equation}
and variance
\begin{align*}
&\Var(\mu_{M,\Omega_k}) \\
&= \frac{1}{\lvert \boxtimes\rvert^2}  \Bigg( \Bigg(\sum_{\substack{\chi \bmod M \\ \chi^2 = \chi_0 \\ \chi\neq \chi_0}}\(m_-(\chi)+\frac{1}{2}\)^{k}\Bigg)^2 \frac{q}{(\sqrt{q} + 1)^2} 
    +\sum_{\alpha_j\neq \pm \sqrt{q}}\Bigg( \sum_{\substack{\chi \bmod M \\ \chi^2 = \chi_0 \\ \chi\neq \chi_0}} m_j^k(\chi) \frac{\lvert\alpha_j\rvert}{\lvert\alpha_j -1\rvert}\Bigg)^2 \Bigg). 
\end{align*}
The results are similar for $\Delta_{\omega_k}(X; M, \square,\boxtimes)$, with $\(m_{\pm}(\chi)+\frac{1}{2}\)$ replaced by $\(m_{\pm}(\chi)-\frac{1}{2}\)$, we denote by $\mu_{M,\omega_k}$ the corresponding limiting distribution.

In the following section we study various square-free polynomials $M$ and we denote by $\chi_M$ the primitive quadratic character modulo $M$.
In the case of prime numbers, it has been observed that the bias tends to disappear as $k\rightarrow\infty$. Moreover in the case of the race with fixed $\Omega$, the bias changes direction with the parity of $k$. Whereas the bias always stays in the direction of the quadratic residues in the race with fixed $\omega$. We present here various examples where this does (or not) happen in the context of irreducible polynomials.

\subsection{Case with no real inverse zero}

In the generic case, we expect that $m_{\pm}(\chi) =0$. 
In particular, for $k$ even, $\mu_{M,\Omega_k}= \mu_{M,\omega_k}$, and, if the non-real zeros are independent of $\pi$, for $k$ odd, $\mu_{M,\Omega_k}$ is the symmetric of $\mu_{M,\omega_k}$ with respect to $0$.
Moreover, for $f = \Omega$ or $\omega$, the mean value of $\mu_{M,f_k}$ becomes negligible as $k$ grows.
This situation is very similar to the case of primes in $\mathbf{Z}$ (see \cite{Meng2017}).
More precisely, we can simplify the expression of the mean value in \eqref{Form_MeanValue}. One has
\begin{align*}
\epsilon_f^k\mathbf{E}\mu_{M,f_k} = \frac{1}{\lvert \boxtimes\rvert}  \sum_{\substack{\chi \bmod M \\ \chi^2 = \chi_0 \\ \chi\neq \chi_0}} \(\frac{1}{2}\)^k \frac{\sqrt{q}}{\sqrt{q}-1}  
= \frac{1}{\lvert \square\rvert}  \frac{1}{2^k} \frac{\sqrt{q}}{\sqrt{q}-1},
\end{align*}
where $\epsilon_{\Omega} = -1$, $\epsilon_{\omega} = 1$.
Note that in this case, $\mathbf{E}\mu_{M,\Omega_k}$ alternates sign as $k$ changes parity and $\mathbf{E}\mu_{M,\omega_k}$ has the same absolute value but stays positive.
Finally, if the sum over the non-real inverse zeros is not empty, one has 
\begin{equation*}
\mathbf{E}\mu_{M,f_k} \ll_M \frac{\sqrt{\Var(\mu_{M,f_k})}}{2^k}.
\end{equation*}
This hints towards a vanishing bias as $k$ gets large, see Proposition~\ref{Prop k limit} for a precise statement.

Let us start with an irreducible polynomial $M$ (as in \cite[Sec.~5]{Cha2008}). 
Assume that the $L$-function $\mathcal{L}(\cdot,\chi_M)$ has only simple zeros that are not real.
Then for $k\geq 1$ we have the formulas
\begin{multline*}
\Delta_{\Omega_k}(X; M, \square,\boxtimes)\\
=(-1)^{k+1} \left\{ \Delta_{\Omega_1}(X; M, \square,\boxtimes)
 +\frac{ 1 - \frac{1}{2^{k-1}} }{2\lvert\square\rvert} \left[ \frac{\sqrt{q}}{\sqrt{q}-1}+(-1)^X \frac{\sqrt{q}}{\sqrt{q}+1} \right] + O_{M}\(\frac{d^k k^2}{\log X} \)\right\},
 \end{multline*}
 \begin{multline*}
\Delta_{\omega_k}(X; M, \square,\boxtimes)\\
=(-1)^{k+1} \left\{ \Delta_{\Omega_1}(X; M, \square,\boxtimes)
 +\frac{1}{2\lvert\square\rvert} \left[ \frac{\sqrt{q}}{\sqrt{q}-1}+(-1)^X \frac{\sqrt{q}}{\sqrt{q}+1} \right]\right\} \\ +\frac{1}{2^k \lvert\square\rvert} \left[ \frac{\sqrt{q}}{\sqrt{q}-1}+(-1)^X \frac{\sqrt{q}}{\sqrt{q}+1} \right] + O_{M}\(\frac{d^k k^2}{\log X} \).
 \end{multline*}
 Note that the term $\frac{ -1 }{2\lvert \square\rvert}  \frac{\sqrt{q}}{\sqrt{q}-1}$ above is the mean value $\mathbf{E}\mu_{M,\Omega_1}$ of the limiting distribution associated to  the function $ \Delta_{\Omega_1}(X; M, \square,\boxtimes)$.

Thus, up to a change of sign, the function $\Delta_{f_k}(\cdot; M, \square,\boxtimes)$ satisfies properties similar to those of the function $\Delta_{\Omega_1}(\cdot; M, \square,\boxtimes)$ regarding the behavior at infinity and the limiting distribution, with the mean value of the limiting distribution going to $0$ as $k$ grows.

\begin{ex}[Bias in the ``wrong direction'']
In \cite[Ex.~5.3]{Cha2008}, Cha studies the polynomial $M = t^5 + 3t^4 + 4t^3 + 2t+ 2 \in \mathbf{F}_{5}[t]$, from his work, we observe that the function\footnote{Note that \cite[Ex.~5.3]{Cha2008} contains a typo, we have $\mathcal{L}(u,\chi_M) = (1 - 2 \sqrt{5} \cos(\tfrac{\pi}{5})u + 5u^2 )(1 - 2 \sqrt{5} \cos(\tfrac{2\pi}{5})u + 5u^2 ).$} 
$$X \mapsto\Delta_{\Omega_1}(X; M, \square,\boxtimes)
 +\frac{1}{2\lvert\square\rvert} \left[ \frac{\sqrt{5}}{\sqrt{5}-1}+(-1)^X \frac{\sqrt{5}}{\sqrt{5}+1} \right]$$ 
 is periodic of period~$10$ and takes positive values larger than $\frac{1}{2\lvert\square\rvert} \left[ \frac{\sqrt{5}}{\sqrt{5}-1}+(-1)^X \frac{\sqrt{5}}{\sqrt{5}+1} \right]$ for~$6$ values of~$X \bmod 10$.
 Thus there is a bias in the ``wrong direction'': one has for all $k\geq 1$ 
 \begin{align*}
     \dens((-1)^k\Delta_{\Omega_k}(\cdot;M,\square,\boxtimes) > 0) = \frac{4}{10}  < \frac12.
 \end{align*}
Contrary to what is expected in the generic case, When $k$ is odd the bias is in the direction of the quadratic residues.
Similarly we obtain that 
 \begin{align*}
     \dens(\Delta_{\omega_1}(\cdot;M,\square,\boxtimes) > 0) = \frac{7}{10}  > \frac12,
 \end{align*}
 and for $k \geq 2$,
  \begin{align*}
     \dens((-1)^{k+1}\Delta_{\omega_k}(\cdot;M,\square,\boxtimes) > 0) = \frac{6}{10}  > \frac12.
 \end{align*}
 In particular the bias changes direction according to the parity of $k$, and when $k$ is odd the bias is in the direction of the quadratic residues.
\end{ex}

As observed in \cite{Li2018}, when $M$ is not irreducible, the $L$-function $\mathcal{L}(\cdot,\chi_M)$ can have non-simple zeros and real zeros. Moreover in Proposition~\ref{Prop Chebyshev many factors}, we obtain extreme biases in races modulo polynomials $M$ with many irreducible factors.
We now focus on square-free non irreducible polynomials. 

\begin{ex}[Dissipating bias in case of double non-real zeros]
\label{Ex_double}

Take $q=5$, and
$M= t^6 + 2t^4 + 3t + 1$ in $\mathbf{F}_5[t]$.
One has
\[\mathcal{L}(u,\chi_M) = (1 + u + 5u^2)^2(1-u) = (1 - 2\sqrt{5}\cos(\theta_1) u + 5u^2)^2 (1-u),\]
where $\theta_1 = \pi + \arctan\sqrt{19}$.
The polynomial $M$ has two irreducible factors of degree $3$ in $\mathbf{F}_{5}[t]$. We denote $M=M_1 M_2$, and for $i=1$, $2$ let $\chi_i$ be the character modulo $M$ induced by the character $\chi_{M_i}$.
We have
\[\mathcal{L}(u,\chi_1) = (1-u +5u^2)(1-u^{3}) = (1 - 2\sqrt{5}\cos(\theta_1 - \pi) u + 5u^2) (1-u^3) \] 
and 
\[\mathcal{L}(u,\chi_2) = (1 +3u +5u^2)(1-u^{3})
= (1 - 2\sqrt{5}\cos(\theta_2) u + 5u^2) (1-u^3), \]
where $\theta_2 = \pi + \arctan(\sqrt{11}/3) $,
 the factor $(1-u^3)$ comes from the fact that the $\chi_i$ are not primitive (see e.g. \cite[Prop.~6.4]{Cha2008}).
Inserting this information in \eqref{Formula_Res_vs_NonRes} 
we obtain
    \begin{multline*}
 \Delta_{\Omega_k}(X; M, \square,\boxtimes)  
   \\ = \frac{(-1)^k}{\lvert \boxtimes\rvert}  \Bigg(\ \frac{3}{2^{k+2}} \(5+ \sqrt{5} + (-1)^{X}  (5-\sqrt{5})\) 
    + 2^{k+1}\re\(\frac{10}{11 + i\sqrt{19}} e^{iX\theta_1}\) \\ + 
    2\re\(\frac{10}{9+ i\sqrt{19}}  	
e^{iX(\theta_1 -\pi)}\)+
2\re\(\frac{10}{13 + i\sqrt{11}} e^{iX \theta_2}\) \Bigg) + O_{M}\(\frac{6^k k^2}{\log X}\) .
	\end{multline*}

We observe that $\theta_1$ is not a rational multiple of $\pi$.
This follows from the fact that for any $n\in\mathbf{N}$ the $5$-adic valuation of $\cos(n\theta_1)$ is $-n/2$, thus we cannot have $\cos(n\theta_1)=\pm 1$ except for $n=0$.
 Hence by Proposition~\ref{Prop_limitingDist}.\ref{Item_continuous}, 
for each $k\geq 1$,    the corresponding limiting distribution is continuous.
Moreover it has mean value $ \mathbf{E} \asymp\frac{(-1)^k}{2^{k}(k-1)!} $ and variance $\Var \asymp\frac{2^{2k}}{(k-1)!^2}$.

Note that LI is not satisfied in this example. However, Damien Roy and Luca Ghidelli observed that the set $\lbrace \pi, \theta_1,\theta_2 \rbrace$ is linearly independent over $\mathbf{Q}$. For any $(a,b,c) \in \mathbf{Z}^3$, we see using the Chebyshev polynomials of the second kind that  $\sin(a\pi + b\theta_1) \in \sqrt{19}\mathbf{Q}(\sqrt{5})$ and $\sin(c\theta_2) \in \sqrt{11}\mathbf{Q}(\sqrt{5})$, hence the only chance for them to be equal is to be $0$.

\begin{table}
\begin{tabular}{|r|c|c|}
\hline
 $k$ &  $\#\{ X \leq 10^9 : {\Delta}_{\Omega_k}(X; M, \square,\boxtimes) >0 \} $& $\#\{ X \leq 10^9 : {\Delta}_{\omega_k}(X; M, \square,\boxtimes) >0 \} $\\
\hline

1& $194\ 355\ 543$ & $805\ 644\ 606$\\ 

2 & $563\ 506\ 459$ &$ 563\ 506\ 459 $\\

3 & $484\ 542\ 923$ &$515\ 457\ 280$\\

4 &  $503\ 903\ 947$ &$503\ 903\ 947$   \\
5 & $499\ 014\ 553$ & $500\ 985\ 439$\\

6 & $500\ 247\ 844$ & $500\ 247\ 844 $ \\

7 & $499\ 937\ 823$ & $500\ 062\ 193$\\

8 & $ 500\ 015\ 580$ & $500\ 015\ 580$ \\

9 & $499\ 996\ 073$ &$500\ 003\ 876$ \\

10 & $500\ 000\ 986$ &  $500\ 000\ 986$ \\
\hline
\end{tabular}

\caption{Approximation of the bias of $\Delta_{\Omega_k}$ and $\Delta_{\omega_k}$ for $k \in \lbrace 1,\ldots, 10\rbrace$ }\label{Table_Ex2}
\end{table}

We observe that the term $2^{k+1}\re\(\frac{10}{11 + i\sqrt{19}} e^{iX\theta_1}\)$ will become the leading term as $k$ grows. This term corresponds to a symmetric distribution with mean value equal to zero. Proposition~\ref{Prop k limit} predicts that the bias tends to $\frac{1}{2}$ as $k$ grows.
We observe this tendency in the data;
in Table~\ref{Table_Ex2} we present an approximation of the bias for the functions $ \Delta_{f_k}(X; M, \square,\boxtimes) \bmod o(1) $, with $f = \Omega$ or $\omega$,
computed for $1\leq X \leq 10^9$ and $1\leq k\leq 10$.
\end{ex}

\subsection{Case where $\sqrt{q}$ or $-\sqrt{q}$ is an inverse zero.}\label{subsec_examples_realZero}

In \cite{Li2018}, Li showed the existence of a family of polynomials $M$ satisfying $m_{+}(\chi_M) >0$.
We now use some of these polynomials to obtain completely biased races between quadratic residues and non-quadratic residues.

\begin{ex}[Complete bias in case of a zero at $\tfrac{1}{2}$]
\label{Ex_sqrt(q)}

Taking $q=9$, we study polynomials with coefficients in $\mathbf{F}_{9}= \mathbf{F}_{3}[a]$ (i.e. $a$ is a generator of $\mathbf{F}_9$ over $\mathbf{F}_3$).
Let
$M= t^4 + 2t^3 + 2t + a^7$.
This polynomial is square-free and has the particularity that $m_{+}(\chi_{M}) =2$ where $\chi_{M}$ is the primitive quadratic character modulo $M$ (see \cite{Li2018}).
More precisely,
\[\mathcal{L}(u,\chi_M) = (1 - 3u)^2.\]

The polynomial $M$ has two irreducible factors of degree $2$ in $\mathbf{F}_{9}[t]$. We denote $M=M_1 M_2$, and for $i=1$, $2$, let $\chi_i$ be the character modulo $M$ induced by the character $\chi_{M_i}$.
Then for $i = 1$, $2$, one has 
\[\mathcal{L}(u,\chi_i) = (1-u)(1-u^{2}).\]
In particular, the only inverse zero of a quadratic character modulo $M$ with norm $\sqrt{9} = 3$ is the real zero $\alpha = 3$ with multiplicity $2$. 
Inserting this information in \eqref{Formula_Res_vs_NonRes} and \eqref{Formula_Res_vs_NonRes_littleomega}, we obtain
    \begin{equation*}
   \Delta_{\Omega_k}(X; M, \square,\boxtimes)  
    = \frac{(-1)^k}{\lvert \boxtimes\rvert} \Bigg\{ \frac{2 + 5^k}{2^k} \frac{3}{2} +\frac{3}{2^k} \frac{3}{4}(-1)^{X}  \Bigg\}
     + O_{M}\(\frac{4^k k^2}{\log X}\),
	\end{equation*}
	and
	    \begin{equation*}
  \Delta_{\omega_k}(X; M, \square,\boxtimes)  
    = \frac{1}{\lvert \boxtimes\rvert} \Bigg\{ \frac{2 + (-3)^k}{2^k} \frac{3}{2} +\frac{3}{2^k} \frac{3}{4}(-1)^{X}  \Bigg\}
     + O_{M}\(\frac{4^k k^2}{\log X}\).
	\end{equation*}
In each case, for each $k\geq 1$,    the limiting distribution is a sum of two Dirac deltas, symmetric with respect to the mean value.
One can observe that, in each case and for any $k\geq 2$, the constant term is larger in absolute value than the oscillating term.
We deduce that, for $k\geq 2$, 
\[\dens ((-1)^k \Delta_{\Omega_k}(\cdot;M,\square,\boxtimes) > 0)
= \dens ((-1)^k \Delta_{\omega_k}(\cdot;M,\square,\boxtimes) > 0)= 1.
\]
We say that the bias is complete.
Note that in this case, contrary to the case of prime numbers, when $k$ is odd, the function~$\Delta_{\omega_k}(\cdot;M,\square,\boxtimes)$ does not have a bias towards quadratic residues.
\end{ex}

\begin{Rk}
We note that the complete bias obtained in Example~\ref{Ex_sqrt(q)} could be one of the simplest ways to observe such a phenomenon.
Previously, in the setting of prime number races, Fiorilli \cite{Fiorilli_HighlyBiased} observed that arbitrary large biases could be obtained in the race between quadratic residues and non-quadratic residues modulo an integer with many prime factors (see also Proposition~\ref{Prop Chebyshev many factors} for a translation in our setting).
Fiorilli's large bias is due to the squares of prime numbers.
Note that over number fields, the infinity of zeros of the $L$-functions is (under the GRH) an obstruction to the existence of complete biases in prime number races with positive coefficients (see \cite[Rk.~2.5]{RS}).
The first observation of a complete bias is in \cite[Th.~1.5]{CFJ} in the context of Mazur's question on Chebyshev's bias for elliptic curves over function fields.
As in \cite{CFJ}, our complete bias is due to a  ``large rank'' i.e. a vanishing of the $L$-function at the central point.
\end{Rk}

\begin{ex}[Absence of bias in case of a zero at $\tfrac{1}{2} + i\pi$]
\label{Ex_-sqrt(q)}
Taking $q=9$, we study polynomials with coefficients in $\mathbf{F}_{9}= \mathbf{F}_{3}[a]$ (as in Example~\ref{Ex_sqrt(q)}).
Let
$M= t^3 -t$.
This polynomial is square-free and has the particularity that $m_{-}(\chi_{M}) =2$.
More precisely,
\[\mathcal{L}(u,\chi_M) = (1 + 3u)^2.\]

The polynomial $M$ has three irreducible factors of degree $1$ in $\mathbf{F}_{9}[t]$. We denote $M=M_1 M_2 M_3$, and for $i=1$, $2$, $3$ let $\chi_i$ be the character modulo $M$ induced by the character $\chi_{M_i}$.
For $i\neq j \in \lbrace 1,2,3\rbrace$ one has
\begin{equation*}
\mathcal{L}(u,\chi_i) = (1-u)^2, \quad \text{ and } \quad 
\mathcal{L}(u,\chi_i\chi_j) = (1-u)^2.
\end{equation*}
In particular, the only inverse zero of a quadratic character modulo $M$ with norm $\sqrt{9} = 3$ is the real zero $\alpha = -3$ with multiplicity $2$. 
Inserting this information into \eqref{Formula_Res_vs_NonRes} and \eqref{Formula_Res_vs_NonRes_littleomega}, we obtain
    \begin{equation*}
\Delta_{\Omega_k}(X; M, \square,\boxtimes)  
    = \frac{(-1)^k}{\lvert \boxtimes\rvert} \Bigg\{ \ \frac{7}{2^k} \frac{3}{2} + \frac{6 + 5^k}{2^k} \frac{3}{4}(-1)^{X} \Bigg\}
     + O_{M}\(\frac{3^k k^2}{\log X}\) ,
	\end{equation*}
	and 
	    \begin{equation*}
   \Delta_{\omega_k}(X; M, \square,\boxtimes)  
    = \frac{1}{\lvert \boxtimes\rvert} \Bigg\{ \ \frac{7}{2^k} \frac{3}{2} + \frac{6 + (-3)^k}{2^k} \frac{3}{4}(-1)^{X} \Bigg\}
     + O_{M}\(\frac{3^k k^2}{\log X}\). 
	\end{equation*}
In each case, the limiting distribution associated to the function for each fixed $k$ is again a sum of two Dirac deltas, symmetric with respect to the mean value.
We observe that for $k =1$ the constant term dominates the sign of the function so there are complete biases $\dens(\Delta_{\Omega_1}(X;M,\square,\boxtimes) > 0) = 0$ and $\dens(\Delta_{\omega_1}(X;M,\square,\boxtimes) > 0) = 1$.
For $k\geq 2$, the two Dirac deltas are each on one side of zero, hence
\[
\dens(\Delta_{\Omega_k}(\cdot;M,\square,\boxtimes) > 0) = \dens(\Delta_{\omega_k}(\cdot;M,\square,\boxtimes) > 0)= \frac{1}{2},
\]
the race is unbiased.
\end{ex}

The examples in this section illustrate the following more general result, we can always find unbiased and completely biased races.

\begin{prop}
    \label{Prop using Honda Tate}
    Let $\mathbf{F}_{q}$ be a finite field of odd characteristic, then there exists $M_{\frac12}\in \mathbf{F}_{q}[t]$ such that,
    for $f= \Omega$ or $\omega$, and for $k$ large enough, one has
    $$ \dens(\Delta_{f_k}(\cdot;M_{\frac12},\square,\boxtimes)>0) = \frac12.$$
    Moreover, if $q >3$, there exists $M_{1}\in \mathbf{F}_{q}[t]$ such that,
    for $f= \Omega$ or $\omega$, and for $k$ large enough, one has
    $$ \dens((-1)^k\Delta_{f_k}(\cdot;M_1,\square,\boxtimes)>0) = 1.$$
\end{prop}

\begin{Rk}
It is interesting to note that in the case of an extreme bias, the bias for the function $\Delta_{\omega_k}(\cdot;M,\square,\boxtimes)$ changes direction with the parity of $k$, whereas in the case of integers \cite{Meng2017} the analog function has a bias towards squares independently of the parity of $k$.
\end{Rk}

\begin{proof}
In the case where $q$ is a square,
this result is a consequence of Honda--Tate theorem in the case of elliptic curve: by \cite[Th.~4.1]{Waterhouse} there exist two elliptic curves $E_{\pm}$ on $\mathbf{F}_{q}$ whose Weil polynomial is $P_{\pm}(u) = (1 \mp \sqrt{q}u)^2$.
If $q >3$ is not a square, by \cite[Th.~4.1]{Waterhouse}, there exists one elliptic curve $E_{\frac12}$ on $\mathbf{F}_{q}$ whose Weil polynomial is $P_{\frac12}(u) = 1 + \sqrt{q}u^2$, and by \cite[Th.~1.2]{HNR}, there exist an hyperelliptic curve $C_1$ of genus $2$ whose Weil polynomial is $P_1(u) = (1 - qu^2)^2$.

Since $q$ is odd, using Weierstrass form, each of the elliptic curve $E_{a}$, where $a = +,-$ or $\frac12$, has an affine model with equation $y^2 = M_{a}(x)$, where $M_{a} \in \mathbf{F}_q$ has degree $3$.
Similarly the hyperelliptic curve $C_1$ has an affine model with equation $y^2 = M_{1}(x)$, with $M_1 \in \mathbf{F}_q[t]$ of degree $5$. 
Then $\mathcal{L}(u,\chi_{M_{a}}) = P_{a}(u)$, for $a\in \lbrace +,-,\frac12,1\rbrace$.

For $a\in \lbrace +,-,\frac12\rbrace$, let $D$ be a strict divisor of $M_a$, then $\deg D\leq 2$ so $\mathcal{L}(u,\chi_{D})$ does not have inverse zeros of norm $\sqrt{q}$.
In the case $D$ is a strict divisor of $M_1$, then $\mathcal{L}(u,\chi_{D}) \in \mathbf{Z}[u]$ has at most two inverse zeros of norm $\sqrt{q}$ that are conjugate, in particular, its inverse zeros are simple.

Thus the case where $q$ is a square follows in the same way as in Examples~\ref{Ex_sqrt(q)} and~\ref{Ex_-sqrt(q)}.

In the case where $q$ is not a square, using the information above in~\eqref{Formula_Res_vs_NonRes} we obtain, for $f = \Omega$ or $\omega$,
\begin{multline*}
    \Delta_{f_k}(X; M_{\frac12}, \square,\boxtimes) 
    = \frac{(-1)^k}{\lvert \boxtimes\rvert} \Bigg\{  \( -\tfrac{\epsilon_{f}}{2}\)^k \frac{\sqrt{q}}{\sqrt{q}-1} +\(-\tfrac{\epsilon_f }{2}\)^k \frac{\sqrt{q}}{\sqrt{q} + 1}(-1)^{X} \\
    +2\re\left(  \frac{i\sqrt{q}}{i\sqrt{q} -1} e^{iX\frac{\pi}2}\right)\Bigg\} + O_k\((\log X)^{-1}\), 
	\end{multline*}
where $\epsilon_{\Omega} = -1$, $\epsilon_{\omega} = 1$.
The periodic part inside the brackets takes $4$ different values : $$2q\(\( -\tfrac{\epsilon_{f}}{2}\)^k\tfrac{1}{q-1} \pm \tfrac{1}{q+1}\),  \quad 2\sqrt{q}\(\( -\tfrac{\epsilon_{f}}{2}\)^k\tfrac{1}{q-1} \pm \tfrac{1}{q+1}\), $$
exactly $2$ of them are positive and $2$ are negative when $q>3$ or $k\geq 2$. So the race is unbiased.

Similarly for $M_1$ we have
\begin{multline*}
    \Delta_{f_k}(X; M_1, \square,\boxtimes)  
    = \frac{(-1)^k}{\lvert \boxtimes\rvert} \Bigg\{  \(2 -\epsilon_{f}\tfrac{1}{2}\)^k \frac{\sqrt{q}}{\sqrt{q}-1} +\(2 -\epsilon_f \tfrac{1}{2}\)^k \frac{\sqrt{q}}{\sqrt{q} + 1}(-1)^{X} \\
    +2\re\left( m_1 \frac{\alpha_1}{\alpha_1 -1} e^{iX\gamma_{1}}\right)   + O_k\((\log X)^{-1}\) \Bigg\},
	\end{multline*}
	where $m_1 = 0$ or $1$ and $\alpha_1 = \sqrt{q}e^{i\gamma_1}$ is an inverse zero of $\mathcal{L}(u,\chi_{D})$, for $D$ a strict divisor of $M_1$.
	We observe that the constant term dominates for $k$ large enough; one has an extreme bias: 
$\dens((-1)^k \Delta_{f_k}(X;M,\square,\boxtimes) > 0) = 1,$ with different directions of the bias according to the parity of $k$.
\end{proof}

\section{Limit behaviours}\label{subsec central limit}

In this section we study the limit behaviour of the measures $\mu_{M,f_k}$, for $f= \Omega$ or $\omega$, as $k$ or $\deg M$ gets large.
We present the results by increasing strength of assumption needed.

\subsection{Unconditional results as $k$ grows}
\label{subsec k limit}

First we focus on $k$ getting large while the modulus $M$ is fixed.
We obtain the following unconditional result (see also Remark~\ref{Rk_on_Th_deg<n}.\ref{Item_largestMult}) regarding the $k$-limit of the limiting distributions.

\begin{prop}\label{Prop k limit}
        Fix $M\in \mathbf{F}_{q}[t]$, let $f= \Omega$ or $\omega$, and $\epsilon_{\Omega} = -1$ and $\epsilon_{\omega} = 1$.
        We define
        \[m_{f,\max} = \max_{\chi, j}\lbrace m_{\pm}(\chi) -\epsilon_{f} \frac{1}{2}, m_j(\chi)\rbrace,\]
        where the maximum is taken over all  non-trivial quadratic characters $\chi$ modulo $M$.
        Then, as $k\rightarrow\infty$, the limiting distribution of
        \begin{align*}
    &\Delta_{f_k}^{\mathrm{norm}}(\cdot;M) := \frac{(-1)^k\lvert \boxtimes\rvert \sqrt{q-1}}{m_{f,\max}^{k}\sqrt{q}}\Delta_{f_k}(\cdot;M,\square,\boxtimes)
\end{align*}
converges weakly to some probability measures $\mu_{f,M}$, depending only on the set of zeros of maximal multiplicity.
In particular,
\begin{enumerate}[label = \roman*)]
    \item\label{Item max integer}\emph{(Expected generic case)} if $m_{\Omega,\max}$ is an integer and if the set of zeros of maximal multiplicity generates a symmetric sub-torus, then $\mu_{M,\Omega}=\mu_{M,\omega}$ is symmetric, so the bias dissipates as $k$ gets large;
    \item\label{Item max is m+} if for some non-trivial quadratic character $\chi_1$ modulo $M$ one has
    \[\max_{\chi}\lbrace m_{-}(\chi) \rbrace < m_{+}(\chi_1) = m_{\Omega,\max} - \frac{1}{2}\] 
    then $\mu_{M,\Omega}$ is a Dirac delta, so the bias tends to be extreme as $k$ gets large.
\end{enumerate}
\end{prop}

\begin{Rk}
Note that in the generic case, we expect LI to be satisfied and to have $m_{\Omega,\max}=1$ an integer. So, for most $M \in \mathbf{F}_{q}[t]$ square-free, the bias should dissipate in the race between polynomials with $k$ irreducible factors in the quadratic residues and non-quadratic residues modulo~$M$.
\end{Rk}

\begin{proof}
Let $\varphi_{M,f_k}$ be the Fourier transform of the limiting distribution of $\Delta_{f_k}^{\mathrm{norm}}(\cdot;M)$.
One has
\begin{multline*}
\varphi_{M,f_k}(\xi) 
   = \exp\left(-i\xi \frac{\sqrt{q-1} \sum_{\chi} \left( m_{+}(\chi)  - \epsilon_f \frac{1}{2}  \right)^k}{m_{f,\max}^{k}(\sqrt{q} - 1)} \right) \\ \times 
    \int_{A}\exp\Bigg\lbrace -i\xi\Bigg( \frac{\sqrt{q-1}\sum_{\chi} \left( m_{-}(\chi)  - \epsilon_f \frac{1}{2}  \right)^k}{m_{f,\max}^{k}(\sqrt{q} +1)}(-1)^{a_1} 
  \\+ \sqrt{\frac{q-1}{q}}\sum_{j=2}^{N} 2\frac{m_{j}^k}{m_{\max}^k}\re\left(\frac{\sqrt{q}e^{i\gamma_j}}{\sqrt{q}e^{i\gamma_j} -1}e^{ia_j}\right)  \Bigg)\Bigg\rbrace \diff\omega_{A}(a),
\end{multline*}
where we write the ordered list of inverse-zeros with multiplicities
\begin{equation}\label{Def zeros multiplicities}
   \lbrace (\gamma_{2},m_2), \ldots, (\gamma_{N},m_N)\rbrace = \bigcup\limits_{\substack{\chi \bmod M \\ \chi^{2} = \chi_0 \\ \chi \neq \chi_0}}\lbrace (\gamma,m)\in (0,\pi)\times\mathbf{N}_{>0} : L(\tfrac{1}{2} + i\gamma, \chi) = 0 \text{ with multiplicity } m  \rbrace, 
\end{equation}
and $A$ is the closure of the $1$-parameter group $\lbrace y(\pi,\gamma_2,\ldots,\gamma_N) : y \in \mathbf{Z} \rbrace/2\pi\mathbf{Z}^{N}$.

Now, by dominated convergence theorem, there are four cases according to which zeros have maximal multiplicity. In each case it is easy to see that the limit function $\varphi_{M,f}$ is indeed the Fourier transform of a measure $\mu_{M,f}$, and the conclusion follows by L\'evy's Continuity Theorem.

Suppose that $m_{f,\max}$ is an integer, i.e. the zeros of maximal order are not real.
Up to reordering, we can assume that the first $d$ zeros in \eqref{Def zeros multiplicities} have maximal multiplicity: $m_2 = \ldots = m_d = m_{\max} > \max_{j>d}\lbrace m_j \rbrace$.
We have for $f=\Omega$ or $\omega$ and for every $\xi \in\mathbf{R}$,
\begin{align*}
    \varphi_{M,f}(\xi) := \lim_{k\rightarrow\infty} \varphi_{M,f_k}(\xi) = \int_{A(\max)}
    \exp\left(-i\xi\left( \sqrt{\frac{q-1}{q}}\sum_{j=2}^{d} 2\re\left(\frac{\alpha_j}{\alpha_j -1}e^{ia_j}\right)  \right)\right) \diff\omega_{A(\max)}(a),
\end{align*}
where $A(\max)$ is the closure of the $1$-parameter group $\lbrace y(\gamma_2,\ldots,\gamma_d) : y \in \mathbf{Z} \rbrace/2\pi\mathbf{Z}^{d-1}$.
This follows from the fact that the projection $A\rightarrow A(\max)$ induces a bijection between a sub-torus of $A$ and $A(\max)$, then, by uniqueness of the normalized Haar measure, the measure induced on the sub-torus by the normalized Haar measure of $A$ is exactly the normalized Haar measure of $A(\max)$.

By Proposition~\ref{Prop_limitingDist}.\ref{Item_symmetry},
the function $\varphi_{M,f}$ is even if the sub-torus $A(\max)$ is symmetric. This concludes the proof of Proposition~\ref{Prop k limit}.\ref{Item max integer}.
Note that in the case $m_{\Omega,\max}$ is an integer, we have $m_{\omega,\max} = m_{\Omega,\max}$ and the zeros of maximal order are the same for the two functions $\Delta_{\Omega_k}$ and $\Delta_{\omega_k}$. In particular $\mu_{M,\Omega} = \mu_{M,\omega}$.

In the case the zeros of maximal order are real, we have the three following possibilities.
\begin{enumerate}[label = \roman*)]\addtocounter{enumi}{1}
    \item The maximum $m_{f,\max}$ is reached only by $m_{+}(\chi_1), \ldots, m_{+}(\chi_d)$, for $\chi_1,\ldots,\chi_d$ non-trivial quadratic characters modulo $M$,
    then for every $\xi \in\mathbf{R}$, we have
\begin{align*}
    \varphi_{M,f}(\xi) := \lim_{k\rightarrow\infty} \varphi_{M,f_k}(\xi) = \exp\left(-i\xi \frac{d \sqrt{q-1}}{\sqrt{q} - 1} \right). 
\end{align*}
This is the Fourier transform of a Dirac delta at a positive value, thus \begin{equation*}
    \lim_{k\rightarrow\infty}\dens((-1)^k\Delta_{f_k}(X;M,\square,\boxtimes) >0) = 1.
\end{equation*}
\item The maximum $m_{f,\max}$ is reached only by $m_{-}(\chi_1), \ldots, m_{-}(\chi_d)$, for $\chi_1,\ldots,\chi_d$ non-trivial quadratic characters modulo $M$,
    then for every $\xi \in\mathbf{R}$, we have
\begin{align*}
    \varphi_{M,f}(\xi) := \lim_{k\rightarrow\infty} \varphi_{M,f_k}(\xi) = \cos\left(\xi  \frac{d\sqrt{q-1}}{\sqrt{q} +1}\right). 
\end{align*}
This is the Fourier transform of a combination of two half-Dirac deltas symmetric with respect to~$0$, thus \begin{equation*}
    \lim_{k\rightarrow\infty}\dens(\Delta_{f_k}(X;M,\square,\boxtimes) >0) = \frac{1}{2}.
\end{equation*}
\item The maximum $m_{f,\max}$ is reached by $m_{+}(\chi_1), \ldots, m_{+}(\chi_d)$ and by $m_{-}({\chi'}_1), \ldots, m_{-}({\chi'}_{d'})$ for $\chi_1,\ldots,\chi_d, {\chi'}_1, \ldots, {\chi'}_{d'}$ non-trivial quadratic characters modulo $M$,
    then for every $\xi \in\mathbf{R}$, we have
\begin{align*}
    \varphi_{M,f}(\xi) := \lim_{k\rightarrow\infty} \varphi_{M,f_k}(\xi) = \exp\left(-i\xi \frac{d \sqrt{q-1}}{\sqrt{q} - 1} \right)\cos\left(\xi  \frac{d' \sqrt{q-1}}{\sqrt{q} +1}\right). 
\end{align*}
This is the Fourier transform of a combination of two half-Dirac deltas  symmetric with respect to~$ \frac{d \sqrt{q-1}}{\sqrt{q} - 1}$. 
Thus the limit measure has a complete bias, or is unbiased, or its bias is not well defined depending on whether $\frac{d}{\sqrt{q} -1} > \frac{d'}{\sqrt{q} + 1}$, or $<$, or $=$.
\end{enumerate}
Finally, note that when $m_{\Omega,\max} \notin \mathbf{N}$, the set of zeros of maximal real part for $\Delta_{\Omega_k}$ and $\Delta_{\omega_k}$ can differ, they coincide if $m_{\omega,\max} \notin \mathbf{N}$.
\end{proof}

\subsection{Existence of extreme biases for moduli with many irreducible factors}
\label{subsec extreme bias M limit}

We now keep $k$ fixed and vary the modulus $M$. Following the philosophy of \cite[Th.~1.2]{Fiorilli_HighlyBiased}, we obtain that, as the number of irreducible factors of $M$ increases, extreme biases appear in the race between quadratic and non-quadratic residues modulo $M$.
Thus $1$ is in the closure of the values of the densities in Theorem~\ref{Th central limit for M under LI}.
As in the work of Fiorilli, the full strength of (LI\ding{70}) is not necessary here.

\begin{prop}\label{Prop Chebyshev many factors}
Let $\lbrace M \rbrace$ be a sequence of polynomials in $\mathbf{F}_{q}[t]$ 
such that
the multi-set $\mathcal{Z}(M) = \bigcup\limits_{\substack{\chi \bmod M \\ \chi^{2} = \chi_0 , \chi \neq \chi_0}}\lbrace \gamma \in [0,\pi] :  L(\frac{1}{2} + i\gamma, \chi) = 0  \rbrace$ is linearly independent of $\pi$ over $\mathbf{Q}$. Assume also that the multiplicities of the zeros are bounded : there exists $B>0$ such that  for each $M$, for each $\gamma \in \mathcal{Z}(M)$ one has
$$\sum_{\substack{\chi \bmod M \\ \chi^{2} = \chi_0 , \chi \neq \chi_0}} m_{\gamma}(\chi) \leq B.$$

Then, for $f=\Omega$ or $\omega$, as $\omega(M)\rightarrow\infty$, one has
\begin{equation*}
    \dens((\epsilon_f)^k\Delta_{f_k}(\cdot;M,\square,\boxtimes) >0) \geq 1 - O\left(  \frac{(2B)^{2k} q \deg M}{2^{\omega(M)}} \right),
\end{equation*}
where $\epsilon_{\Omega} = -1$ and $\epsilon_{\omega} = 1$.
\end{prop}

\begin{proof}
The proof follows the idea of \cite[Th.~1.2]{Fiorilli_HighlyBiased}, \cite[Cor.~5.8]{Devin2018}, using Chebyshev's inequality (e.g. \cite[(5.32)]{Billingsley}). However, unlike the results in loc. cit. that use a limiting density for a function over $\mathbf{R}$, we need to be careful about the influence of $\pi$.

Thanks to the hypothesis of linear independence applied to \eqref{Formula_Res_vs_NonRes}, we have that $\mu_{M,f_k}$ is the convolution of two probability measures: $\mu_{M,f_k} = D_{M,f_k} \ast \nu_{M,k}$,
where $D_{M,f_k}$ is the combination of two half-Dirac deltas at $\frac{2(2^{\omega(M)} -1)}{\lvert \boxtimes\rvert}  \(\frac{\epsilon_f}{2}\)^k \frac{\sqrt{q}}{q-1}$  and $\frac{2(2^{\omega(M)} -1)}{\lvert \boxtimes\rvert}  \(\frac{\epsilon_f}{2}\)^k \frac{q}{q-1}$, with $\epsilon_{\Omega} = -1$ and $\epsilon_{\omega} = 1$, 
while $\nu_{M,k}$ has mean value $0$ and variance
\begin{align*}
    \Var(\nu_{M,k}) = \frac{1}{\lvert \boxtimes\rvert^2}  \sum_{\gamma \in \mathcal{Z}(M)}\Bigg( \sum_{\substack{\chi \bmod M \\ \chi^2 = \chi_0 , \chi\neq \chi_0}} m_{\gamma}^{k}(\chi) \frac{\lvert\sqrt{q}e^{i\gamma}\rvert}{\lvert\sqrt{q}e^{i\gamma}-1\rvert}\Bigg)^2 
    \ll \frac{B^{2k}  2^{\omega(M)} \deg M}{\lvert \boxtimes\rvert^2}.
\end{align*}
To obtain this bound, note that there are $2^{\omega(M)} -1$ quadratic characters modulo $M$ (see e.g. \cite[Prop.~1.6]{Rosen2002}) and for each of them the associated Dirichlet~$L$-function is of degree smaller than $\deg M$. (Note also that $\nu_{M,k}$ is independent of $f=\Omega$ or $\omega$.)
Thus, by Proposition~\ref{Prop_limitingDist}.\ref{Item_continuous} and Chebyshev's inequality,
\begin{align*}
    \dens\left((\epsilon_f)^k\Delta_{f_k}(X;M,\square,\boxtimes) >0\right) &\geq \nu_{M,k}\left(\bigg(  \frac{-\sqrt{q}}{(q-1)}\frac{(2^{\omega(M)} -1)}{ 2^{k-1}\lvert \boxtimes\rvert} , \infty\bigg)\right) \\
    &\geq 1 - O\left( \frac{B^{2k}  2^{\omega(M)} \deg M}{\lvert \boxtimes\rvert^2}  \left(\frac{\sqrt{q}}{(q-1)}  \frac{(2^{\omega(M)} -1)}{ 2^{k-1}\lvert \boxtimes\rvert} \right)^{-2} \right) 
\end{align*}
which concludes the proof.
\end{proof}

\subsection{Limit behaviour for moduli satisfying the linear independence}
\label{subsec M limit under LI}

In this section, following~\cite[Sec.~3]{Fiorilli_HighlyBiased}, we generalize \cite[Th.~6.2]{Cha2008} on the central limit behaviour of the measure $\mu_{M,f_k}$ for $f=\Omega$ or $\omega$.
In particular we prove Theorem~\ref{Th central limit for M under LI}; under (LI\ding{70}), the bias can approach any value in $[\tfrac12 ,1]$ as the degree of $M$ gets large.
We already proved that $1$ can be approached by the values of the bias without assuming (LI\ding{70}) in Proposition~\ref{Prop Chebyshev many factors}, thus it remains to prove Theorem~\ref{Th central limit for M under LI} for the interval $[\frac12,1)$.

As noted in Section~\ref{subsec extreme bias M limit} when using Chebyshev's inequality, assuming enough linear independence of the zeros, the distribution $\mu_{M,f_k}$ is well described by the data
\begin{equation*}
    B(M) :=\frac{\lvert \mathbf{E}\mu_{M,f_k}\rvert} {\sqrt{\Var(\nu_{M,k})}}= \frac{(2^{\omega(M)} -1)\sqrt{q}}{2^k(\sqrt{q} -1)}\Bigg(  \sum_{\substack{\chi \bmod M \\ \chi^{2} = \chi_0 \\ \chi \neq \chi_0}}\sum_{j}^{d_{\chi}} \left\lvert \frac{\alpha_j{\chi}}{\alpha_j{\chi} -1} \right\rvert^2  \Bigg)^{-1/2}
\end{equation*}
where $\nu_{M,k}$ is as defined in the proof of Proposition~\ref{Prop Chebyshev many factors}.
By \cite[(44)]{Cha2008} and \cite[Prop.~6.4]{Cha2008}, we have for every non-trivial quadratic character~$\chi$ modulo~$M$,
\begin{align*}
     \sum_{j=1}^{d_{\chi}}\left\lvert\frac{\alpha_j(\chi)}{\alpha_{j}(\chi)-1}\right\rvert^2  = \frac{q}{q-1}(\deg M^*(\chi) -2) + O(\log (\deg M^*(\chi) +1)),
\end{align*}
where $M^*(\chi)$ is the modulus of the primitive character that induces $\chi$.
Note that the sum is empty if $\deg M^*(\chi) \leq 2$.
Thus, summing over the non-trivial quadratic characters we have
\begin{align}\label{bound IM}
     I(M) :=\sum_{\substack{\chi \bmod M \\ \chi^{2} = \chi_0 \\ \chi \neq \chi_0}}
    \sum_{j=1}^{d_{\chi}}  \left\lvert \frac{\alpha_j(\chi)}{\alpha_{j}(\chi)-1}\right\rvert^2
    &= \frac{q}{q-1}\sum_{\substack{D\mid M' \\ D \neq 1}} (\deg D -2) + O(2^{\omega(M)} \log (\deg M' +1))\nonumber\\
    &= \frac{q}{q-1}(2^{\omega(M)} -1) \frac{\deg M' -4}{2} + O(2^{\omega(M)} \log (\deg M' +1),
\end{align}
where $M'$ is the largest square-free divisor of $M$.
Thus, if $\deg M'>4$,
\begin{equation*}
     B(M)= 2^{-k}\frac{\sqrt{2(q-1) (\tau(M') -1)}}{(\sqrt{q} -1)\sqrt{\deg M' -4}} \left( 1 + O\left(\frac{\log(\deg M') }{\deg M'}\right)\right),
\end{equation*}
where $\tau$ counts the number of divisors.
We show that, as $\deg M' \rightarrow\infty$, $B(M)$ can approach any non-negative real number.

\begin{lem}\label{Lem any possible value}
For any fixed $0\leq c <\infty$, there exists a sequence of square-free monic polynomials $M_n \in \mathbf{F}_q[t]$
such that
\begin{equation*}
    \deg M_n \rightarrow\infty \text{ and } \tau(M_n) = c\deg(M_n) + O(1).
    \end{equation*}
\end{lem}

\begin{proof}
The case $c=0$ follows from taking a sequence of irreducible polynomials.
Now, fix $0<c<\infty$, for $\omega$ large enough,
there exist $0<d_1 < d_2 < \ldots < d_{\omega}$ integers such that
\begin{equation*}
    \left[\frac{2^{\omega}}{c}\right] = d_1 + d_2 + \ldots + d_{\omega}.
\end{equation*}
For each $1\leq i \leq \omega$, there exists an irreducible polynomial $P_i \in \mathbf{F}_{q}[t]$ of degree $d_i$.
Then the polynomial $M := P_1 P_2 \ldots P_\omega$ is square-free and satisfies 
\begin{align*}
    \tau(M) = 2^{\omega} = c( \deg M + O(1)).
\end{align*}
This concludes the proof. 
\end{proof}

\begin{Rk}
\label{Rk the measure is as smooth as you want}
Note that we only used the fact that there exist irreducible polynomials of each degree in $\mathbf{F}_{q}[t]$.
In the sequence we constructed, we have $\deg M_n\rightarrow\infty$, thus $I(M_n)\rightarrow \infty$, we deduce that the number of zeros~$\lvert \mathcal{Z}(M_n) \rvert$
gets large too.
In particular, we can always assume that $\lvert \mathcal{Z}(M_n) \rvert\geq 3$
so that the limiting distribution $\mu_{M_n,f_k}$ is absolutely continuous (see \cite[Th.~1.5]{MartinNg}). 
\end{Rk}

In the case of a sequence of polynomials $(M_n)_{n\in \mathbf{N}}$ satisfying (LI\ding{70}) and for which $B(M_n)$ converges, we show that the limiting bias can be precisely described.

\begin{prop}
\label{Prop M limit using Berry Esseen}
Let $b \in [0,\infty),$
    suppose there exist a sequence of polynomials $M_n \in \mathbf{F}_q[t]$ with $\deg M_n \rightarrow\infty$, $B(M_n) \rightarrow b$, and for each $n$, $M_n$ satisfies \emph{(LI\ding{70})}.
    Then, for $f = \Omega$ or $\omega$,  as $\deg M_n \rightarrow\infty$, the limiting distribution $\mu_{M,f_k}^{\mathrm{norm}}$ of 
\begin{align*}
    &\Delta_{f_k}^{\mathrm{norm}}(\cdot;M) := \frac{(\epsilon_f)^k}{\sqrt{\Var\nu_{M,k}}}\Delta_{f_k}(\cdot;M,\square,\boxtimes)
\end{align*}
converges weakly to the distribution $\frac{1}{2}(\delta_{2b/(\sqrt{q} +1)} + \delta_{2b\sqrt{q}/(\sqrt{q} +1)})\ast \mathcal{N}$,
where $\mathcal{N}$ is the standard Gaussian distribution, and where $\epsilon_{\Omega} = -1$ and $\epsilon_{\omega} = 1$. 
More precisely, one has
\begin{equation*}
    \sup_{x\in\mathbf{R}}\left\lvert \int_{-\infty}^{x}\diff\mu_{M,f_k}^{\mathrm{norm}} - \frac{1}{2\sqrt{2\pi}}\int_{-\infty}^{x} e^{-\frac{1}{2}(t - b\frac{2}{\sqrt{q} +1})^2} + e^{-\frac{1}{2}(t - b\frac{2\sqrt{q}}{\sqrt{q} +1})^2} \diff t \right\rvert \ll  2^{-\omega(M_n)}(\deg M'_n)^{-1} + \lvert B(M_n) - b\rvert 
\end{equation*}
where $M'_n$ is the square-free part of $M_n$.
\end{prop}

\begin{proof}
The proof follows ideas from the proof of \cite[Th.~1.1]{Fiorilli_HighlyBiased} and \cite[Th.~4.5]{CFJ}, and is based on the use of Berry--Esseen inequality \cite[Chap.~II, Th.~2a]{Esseen}.
Let $M\in\mathbf{F}_{q}[t]$ be a polynomial satisfying (LI\ding{70}). 
We begin with computing the Fourier transform $\varphi_{M,f_k}$ of the limiting distribution $\mu^{\mathrm{norm}}_{M,f_k}$ of $\frac{(\epsilon_f)^k}{\sqrt{\Var \nu_{M,k}}}\Delta_{k}$ using Proposition~\ref{Prop_limitingDist}.\ref{Item_Fourier} for \eqref{Formula_Res_vs_NonRes} where we assume linear independence.
With the notations of~\eqref{Not_L-function}, one has
\begin{align*}
   \varphi_{M,f_k}(\xi) &= \hat\mu_{M,f_k}\left(\frac{(\epsilon_f)^k}{\sqrt{\Var \nu_{M,k}}}\xi\right) \\
   &=  \exp\left(-i B(M)\xi \right)
   \cos\left( \frac{\sqrt{q} - 1}{\sqrt{q} + 1} B(M)\xi \right) 
     \prod_{\substack{\chi \bmod M \\ \chi^{2} = \chi_0 \\ \chi \neq \chi_0}}\prod_{j=1}^{d_{\chi}/2}J_{0}\Big(  2 I(M)^{-1/2}
    \left\lvert \frac{\alpha_j(\chi)}{\alpha_j(\chi) -1} \right\rvert (-\epsilon_f)^k\xi \Big),
\end{align*}
where using the functional equation, we assume that the first $\frac{d_{\chi}}{2}$ non-real inverse zeros have positive imaginary part, (recall that, since $\chi$ is real, up to reordering we can write $\alpha_{d_{\chi} - j}(\chi) = \overline{\alpha_{j}(\chi)}$ for $j\in \lbrace1,\ldots,d_{\chi}\rbrace$).

Using the parity of the Bessel function and the power series expansion $\log J_0(z) = - \frac{z^2}{4} + O(z^4)$, for~$\lvert z\rvert <\frac{12}{5}$ (see e.g. \cite[Lem.~2.8]{FiorilliMartin}), we get that, for any $\lvert \xi\rvert < \frac{1}{2} I(M)^{1/2}$,
\begin{multline}
    \label{Asympt Fourier deg M}
    \log \left\lbrace\varphi_{M,f_k}(\xi) \exp\left(i B(M)\xi \right)
   \cos\left( \tfrac{\sqrt{q} - 1}{\sqrt{q} + 1} B(M)\xi \right)^{-1} \right\rbrace  \\
    = - \frac{1}{2}\xi^2 
    + O\Big( I(M)^{-2}\sum_{\substack{\chi \bmod M \\ \chi^{2} = \chi_0 \\ \chi \neq \chi_0}}
    \sum_{j=1}^{d_{\chi}/2}  \left\lvert \frac{\alpha_j(\chi)}{\alpha_{j}(\chi)-1}\right\rvert^4 \xi^4  \Big).
\end{multline}
Since $\left\lvert \frac{\alpha_j(\chi)}{\alpha_{j}(\chi)-1}\right\rvert \leq \frac{\sqrt{q}}{\sqrt{q} - 1}$, the error term in~\eqref{Asympt Fourier deg M} is $O( \xi^4 I(M)^{-1})$.
In the other direction, in the range $\lvert \xi \rvert > \tfrac{1}{2}I(M)^{\frac14}$, we have  
\begin{align*}
     2 I(M)^{-1/2}
    \left\lvert \frac{\alpha_j(\chi)}{\alpha_j(\chi) -1} \right\rvert \lvert\xi\rvert > I(M)^{-1/4}
    \left\lvert \frac{\alpha_j(\chi)}{\alpha_j(\chi) -1} \right\rvert  \in [0,\tfrac{5}{3}]
\end{align*}
for $I(M)$ large enough.
Since $J_0$ is positive and decreasing on the interval $[0,\tfrac{5}{3}]$ and that for all $z \geq \tfrac{5}{3}$ we have $\lvert J_0(z) \rvert \leq J_0(\tfrac{5}{3})$, we deduce
\begin{multline} \label{Asympt Fourier xi large}
    \log \left\lbrace\varphi_{M,f_k}(\xi) \exp\left(i B(M)\xi \right)
   \cos\left( \tfrac{\sqrt{q} - 1}{\sqrt{q} + 1} B(M)\xi \right)^{-1} \right\rbrace  \\ \leq \sum_{\substack{\chi \bmod M \\ \chi^{2} = \chi_0 \\ \chi \neq \chi_0}}\sum_{j=1}^{d_{\chi}/2}\log J_{0}\Big( I(M)^{-1/4}
    \left\lvert \frac{\alpha_j(\chi)}{\alpha_j(\chi) -1} \right\rvert \Big) 
    = - \frac{1}{4} I(M)^{\frac12} 
    + O(1).
\end{multline}

Note that \eqref{Asympt Fourier deg M} is enough to show, by L\'evy's Continuity Theorem, that $\nu_{M,k}^{\mathrm{norm}}$ converges weakly to the standard Gaussian distribution as $I(M) \rightarrow\infty$.
Since limit and convolution are compatible, 
we deduce that if $B(M)\rightarrow b$ converges, then $\mu_{M,f_k}^{\mathrm{norm}}$ converges weakly to a distribution that is a sum of two Gaussian distributions centered at $b(1 \pm \tfrac{\sqrt{q} -1}{\sqrt{q} +1})$.

 The precise rate of convergence of the distribution function is obtained via the Berry--Esseen inequality \cite[Chap.~II, Th.~2a]{Esseen}.
 Let $F$ and $G$ be the cumulative distribution functions of $\frac{1}{2}(\delta_{2b/(\sqrt{q} +1)} + \delta_{2b\sqrt{q}/(\sqrt{q} +1)})\ast \mathcal{N}$ and $\mu_{M,f_k}^{\mathrm{norm}}$, precisely:
 \begin{equation*}
     F(x) = \frac{1}{2\sqrt{2\pi}}\int_{-\infty}^{x}( e^{-\frac{1}{2}(t - b\frac{2}{\sqrt{q} +1})^2} + e^{-\frac{1}{2}(t - b\frac{2\sqrt{q}}{\sqrt{q} +1})^2} )\diff t 
     \quad \text{ and } \quad G(x) = \int_{-\infty}^{x}\diff\mu_{M,f_k}^{\mathrm{norm}}.
 \end{equation*}
 As observed in Remark~\ref{Rk the measure is as smooth as you want}, when $\deg M$ is large enough, the function $G$ is differentiable.
For any $T>0$, we have
 \begin{equation}\label{Eq Berry Esseen}
     \lvert G(x) - F(x) \rvert \ll \int_{-T}^{T}\left\lvert \frac{\varphi_{M,f_k}(\xi) - \exp(-ib\xi) \cos\left( \tfrac{\sqrt{q} - 1}{\sqrt{q} + 1} b\xi \right)e^{-\frac{1}{2}\xi^2}}{\xi} \right\rvert\diff\xi
     + \frac{\lVert G' \rVert_{\infty}}{T}. 
 \end{equation}
 Let us first estimate the second term in the right-hand side of \eqref{Eq Berry Esseen}.
 We have for all $x\in \mathbf{R}$
 \begin{align*}
     G'(x) = \frac{1}{2\pi}\int_{\mathbf{R}}e^{-ix\xi}\varphi_{M,f_k}(\xi)\diff\xi \ll \int_{\mathbf{R}}  \prod_{\substack{\chi \bmod M \\ \chi^{2} = \chi_0 \\ \chi \neq \chi_0}}\prod_{j=1}^{d_{\chi}/2}\left\lvert J_{0}\Big(  2 I(M)^{-1/2}
    \left\lvert \frac{\alpha_j(\chi)}{\alpha_j(\chi) -1} \right\rvert \xi \Big) \right\rvert\diff\xi.
 \end{align*}
 Using the bound $\lvert J_0(z) \rvert \ll \min(1,x^{-\frac12})$, and that $\lvert \mathcal{Z}(M) \rvert \geq 3$ we obtain that
 \begin{equation}\label{Bound BerryEsseen derivative}
     \lVert G' \rVert_{\infty} \ll  \int_{\mathbf{R}} \min(1, I(M)^{3/4}\lvert\xi\rvert^{-3/2}) 
     \diff\xi \ll I(M)^{3/4}.
 \end{equation}
 To bound the integral in \eqref{Eq Berry Esseen}, we cut the interval of integration in two ranges.
 First, by \eqref{Asympt Fourier deg M}, the integral in the range $\lvert\xi\rvert \leq \frac{1}{2}I(M)^{1/4}$ is
 \begin{align}\label{Ineq Berry--Esseen first range}
     \int_{-\frac{1}{2}I(M)^{1/4}}^{\frac{1}{2}I(M)^{1/4}}
     \frac{e^{-\frac{1}{2}\xi^2}} {2\lvert\xi\rvert} &
     \left\lvert \sum_{\pm} e^{-ib(1 \pm \tfrac{\sqrt{q} - 1}{\sqrt{q} + 1})\xi} \left( 1 - \exp\left(-i(B(M) - b)(1 \pm \tfrac{\sqrt{q} - 1}{\sqrt{q} + 1})\xi + O(\xi^4 I(M)^{-1})\right)\right) \right\rvert
     \diff\xi \nonumber\\
     \ll&
     \int_{-\frac{1}{2}I(M)^{1/4}}^{\frac{1}{2}I(M)^{1/4}}
     e^{-\frac{1}{2}\xi^2}\left( \lvert B(M) - b \rvert + \lvert \xi\rvert^3 I(M)^{-1} \right)
     \diff\xi \\
     \ll& \lvert B(M) - b \rvert +  I(M)^{-1}.\nonumber
 \end{align}
 In the range $\frac{1}{2}I(M)^{1/4} \leq \lvert\xi\rvert \leq T$, we use the bound from~\eqref{Asympt Fourier xi large}:
 \begin{align}\label{Ineq Berry--Esseen second range}
     \int_{\frac{1}{2}I(M)^{1/4} \leq \lvert\xi\rvert \leq T}&
     \left\lvert \frac{\varphi_{M,f_k}(\xi) - \exp(-ib\xi) \cos\left( \tfrac{\sqrt{q} - 1}{\sqrt{q} + 1} b\xi \right)e^{-\frac{1}{2}\xi^2}}{\xi} \right\rvert\diff\xi \\
     \ll& \int_{\frac{1}{2}I(M)^{1/4} \leq \lvert\xi\rvert \leq T} e^{-\frac14 I(M)^{\frac12}} \frac{\diff\xi}{\lvert \xi \rvert} + 
     \int_{\frac{1}{2}I(M)^{1/4} \leq \lvert\xi\rvert \leq T}
     e^{-\frac{1}{2}\xi^2}\frac{\diff\xi}{\lvert \xi\rvert }\nonumber \\
     \ll& e^{-\frac14 I(M)^{\frac12}}\log T + e^{-\frac18 I(M)^{\frac12}}.\nonumber
\end{align}
Now combining \eqref{Bound BerryEsseen derivative}, \eqref{Ineq Berry--Esseen first range} and \eqref{Ineq Berry--Esseen second range} in \eqref{Eq Berry Esseen} for $T= I(M)^{\frac74}$, and the estimate~\eqref{bound IM} gives the result, as $\deg M\rightarrow\infty$.
\end{proof}

The proof of Theorem~\ref{Th central limit for M under LI} follows.
\begin{proof}[Proof of Theorem~\ref{Th central limit for M under LI}]
Let $\eta \in [\tfrac12,1]$. If $\eta = 1$, then by Proposition~\ref{Prop Chebyshev many factors}, there exist a sequence of polynomials $M\in \mathbf{F}_q[t]$ with $\dens((\epsilon_f)^k\Delta_{f_k}(\cdot;M,\square,\boxtimes)) \rightarrow \eta$ as $\deg M \rightarrow \infty$.

Now assume $\eta \in [\frac12,1)$. Since the function $b \rightarrow \frac{1}{2\sqrt{2\pi}}\int_{0}^{\infty}( e^{-\frac{1}{2}(t - b\frac{2}{\sqrt{q} +1})^2} + e^{-\frac{1}{2}(t - b\frac{2\sqrt{q}}{\sqrt{q} +1})^2} )\diff t $ is increasing continuous and taking values in $[\tfrac12,1)$ when $b\in [0,\infty)$, there exist a unique $b$ such that  $\frac{1}{2\sqrt{2\pi}}\int_{0}^{\infty}( e^{-\frac{1}{2}(t - b\frac{2}{\sqrt{q} +1})^2} + e^{-\frac{1}{2}(t - b\frac{2\sqrt{q}}{\sqrt{q} +1})^2} )\diff t = \eta$. Lemma~\ref{Lem any possible value} ensures the existence of a sequence of square-free monic polynomials $M$ with
\begin{align*}
    B(M) = \begin{cases} b\left(1 + O\left(\frac{\log\deg M}{\deg M}\right)\right) \text{ if } b > 0,\\
    O\left( 2^{-k} 2^{\omega(M)/2} (\deg M)^{-\frac12} \right) \text{ if } b=0,
    \end{cases}
\end{align*} as $\deg M\rightarrow\infty$. Those polynomials are only defined by their degree and number of divisors, according to the hypothesis of Theorem~\ref{Th central limit for M under LI}, we can assume that each of them satisfies (LI\ding{70}).
Then applying Proposition~\ref{Prop M limit using Berry Esseen} to this sequence, we get
\begin{equation*}
    \lvert \mu_{M,f_k}^{\mathrm{norm}}([0,\infty)) - \eta\rvert \ll \begin{cases} (2^{-\omega(M)} + \log \deg M)(\deg M)^{-1} \text{ if } b >0, \\
    2^{-\omega(M)}(\deg M)^{-1} + 2^{-k} 2^{\omega(M)/2} (\deg M)^{-\frac12} \text{ if } b=0.
    \end{cases}
\end{equation*}
Since we assume (LI\ding{70}) for $M$, one has $$\mu_{M,f_k}^{\mathrm{norm}}([0,\infty)) = \dens ((\epsilon_f)^k \Delta_{k}(\cdot;M,\square,\boxtimes) > 0),$$
this concludes the proof.
\end{proof}

\section{Character sums over polynomials of degree $n$ with $k$ irreducible factors}\label{Sec_proof_deg=n}

For $k\geq 1$, $\chi$ a Dirichlet character modulo $M$, and $f = \Omega$ or $\omega$ we define
\begin{equation}\label{Eq defi pi(chi)}
  \pi_{f_k}(n, \chi)=\sum_{\substack{N \text{ monic, } (N,M)=1\\ \deg(N) = n,~ f(N)=k }} \chi(N).  
\end{equation}
In this section, we prove the following result about the asymptotic expansion of $\pi_{f_k}(n,\chi)$ by induction over the number of irreducible factors $k$.
\begin{theo}\label{Prop_k_general}
Let $M \in \mathbf{F}_{q}[t]$ of degree $d \geq 1$.
Let $k$ be a positive integer.
 Let $\chi$ be a non-trivial Dirichlet character modulo $M$, and
 $$\gamma(\chi) = \min\limits_{1\leq i\neq j\leq d_{\chi}}\(\lbrace \lvert \gamma_i(\chi) - \gamma_j(\chi) \rvert, \lvert \gamma_{i}(\chi) \rvert, \lvert \pi -\gamma_{i}(\chi) \rvert \rbrace\).$$
With notations as in \eqref{Not_L-function}, for $f = \Omega$ or $\omega$, under the conditions $k =o((\log n)^{\frac{1}{2}})$, and $q\geq 5$ if $f= \omega$, one has 
\begin{align*}
\pi_{f_k}(n,\chi) =  \frac{(-1)^k}{(k-1)!} & \Bigg\{ 
 \( \(m_+(\chi) -\epsilon_f\frac{\delta(\chi^2)}{2}\)^k+(-1)^n \(m_-(\chi) -\epsilon_f\frac{\delta(\chi^2)}{2}\)^k \)\frac{q^{n/2}(\log n)^{k-1}}{n} \\
+ &\sum_{\alpha_j\neq\pm\sqrt{q}} m_j(\chi)^k \frac{\alpha_j^n(\chi) (\log n)^{k-1}}{n} +O \left( d^k \frac{k(k-1)}{\gamma(\chi)} \frac{q^{n/2}(\log n)^{k-2}}{n} + d\frac{q^{n/3}}{n} \right) \Bigg\},
\end{align*}
where the implicit constant is absolute, $\delta(\chi^2) = 1$ if $\chi^2 = \chi_0$ and $0$ otherwise, $\epsilon_{\Omega} = -1$ and $\epsilon_{\omega} = 1$.

 \end{theo}

\subsection{Case $k=1$}

We start by recalling the usual setting of the race between irreducible polynomials which is the base case in our induction.
In this situation we obtain a better error term.
\begin{prop}\label{Prop_k=1}
	Let $\chi$ be a non-trivial Dirichlet character modulo $M$.
    Its Dirichlet $L$-function $\mathcal{L}(u,\chi)$ is a polynomial, 
	let $\alpha_{1}(\chi), \ldots, \alpha_{d_{\chi}}(\chi)$ denote the distinct non-real inverse zeros of norm $\sqrt{q}$ of $\mathcal{L}(u,\chi)$, and $m_{1},\ldots,m_{d_{\chi}} \in \mathbf{Z}_{>0}$ be their multiplicities.
	For $f=\Omega$ or $\omega$, one has
	\begin{align*}
	 \pi_{f_1}(n,\chi)& 
    =-\sum_{j=1}^{d_{\chi}} m_{j}(\chi)\frac{\alpha_{j}(\chi)^n}{n}  - \(-\epsilon_f \delta\left(\frac{n}{2},\chi^{2}\right) +m_{+}(\chi) +(-1)^{n}m_{-}(\chi)\)\frac{q^{n/2}}{n} + O(\frac{dq^{n/3}}{n}),
	\end{align*}
	where   $\epsilon_{\Omega} = -1$, $\epsilon_{\omega} =1$, and $$\delta\left(\frac{n}{2},\chi^{2}\right) =\begin{cases}
    1, & \text{if} ~n~\text{is even and} ~\chi^{2} = \chi_{0};\\
    0, & \text{otherwise.}
   \end{cases} $$ 
\end{prop}

\begin{proof}
	We write the Dirichlet $L$-function in two different ways.
	First it is defined as an Euler product:
	\begin{align*}
	\mathcal{L}(u,\chi) = \prod_{n=1}^{\infty} \prod_{\substack{P \text{ irred.}\\ \deg(P) = n \\ P\nmid M}} (1- \chi(P)u^n)^{-1}.
	\end{align*}
	As $\chi \neq \chi_{0}$, the function $\mathcal{L}(u,\chi)$ is a polynomial in $u$, using the notations of \eqref{Not_L-function},
\begin{align*}
	\mathcal{L}(u,\chi) = (1-\sqrt{q}u)^{m_+}(1+\sqrt{q}u)^{m_-}\prod_{j=1}^{d_{\chi}} (1- \alpha_{j}(\chi)u)^{m_{j}} \prod_{j'=1}^{d'_{\chi}} (1- \beta_{j'}(\chi)u),
	\end{align*}
    where $\lvert \beta_{j}(\chi)\rvert =1$. 
    By comparing the coefficients of degree $n$ in the two expressions of the logarithm we obtain
    	\begin{align*}
	\sum_{\ell\mid n} \frac{\ell}{n}\sum_{\substack{P \text{ irred.}\\ \deg(P) = \ell \\ P\nmid M}} \chi(P)^{n/\ell}  &= -\frac{q^{n/2}}{n}(m_+ +(-1)^{n}m_-)  -\sum_{j=1}^{d_\chi} m_{j}\frac{\alpha_{j}(\chi)^n}{n} - \sum_{j'=1}^{d'_\chi}\frac{\beta_{j'}(\chi)^n}{n}.
	\end{align*}
	
	Thus
	\begin{align*}
	\pi_{\Omega_1}(n,\chi) &= -\frac{q^{n/2}}{n}(m_+ +(-1)^{n}m_-)  -\sum_{j=1}^{d_\chi} m_{j}\frac{\alpha_{j}(\chi)^n}{n} + O\left(\frac{d'_{\chi}}{n}\right) - \sum_{\substack{\ell\mid n \\ \ell\neq n}} \frac{\ell}{n}\sum_{\substack{P \text{ irred.}\\ \deg(P) = \ell \\ P\nmid M}} \chi(P)^{n/\ell} \\ 
	&=-\frac{q^{n/2}}{n}(m_+ +(-1)^{n}m_-)  -\sum_{j=1}^{d_\chi} m_{j}\frac{\alpha_{j}(\chi)^n}{n}  - \frac{1}{2}\pi_{\Omega_1}\left(\frac{n}{2},\chi^{2}\right) + O\left(\frac{d + q^{n/3}}{n}\right),
	\end{align*}
	and
		\begin{align*}
	\pi_{\omega_1}(n,\chi) &= -\frac{q^{n/2}}{n}(m_+ +(-1)^{n}m_-)  -\sum_{j=1}^{d_\chi} m_{j}\frac{\alpha_{j}(\chi)^n}{n} + O\left(\frac{d'_{\chi}}{n}\right) + \sum_{\substack{\ell\mid n \\ \ell\neq n}} (1 -\frac{\ell}{n})\sum_{\substack{P \text{ irred.}\\ \deg(P) = \ell \\ P\nmid M}} \chi(P)^{n/\ell} \\ 
	&=-\frac{q^{n/2}}{n}(m_+ +(-1)^{n}m_-)  -\sum_{j=1}^{d_\chi} m_{j}\frac{\alpha_{j}(\chi)^n}{n}  + \frac{1}{2}\pi_{\Omega_1}\left(\frac{n}{2},\chi^{2}\right) + O\left(\frac{d+q^{n/3}}{n}\right),
	\end{align*}
	
	where $\pi_{\Omega_1}\left(\frac{n}{2},\chi^{2}\right) =0$ if $n$ is odd, and it can be included in the error term if $\chi^{2} \neq \chi_{0}$.
	If $n$ is even and $\chi^{2} = \chi_{0}$, one has (\cite[Th.~2.2]{Rosen2002})
	$$\pi_{\Omega_1}\left(\frac{n}{2},\chi^{2}\right) = 2\frac{q^{n/2}}{n} + O(q^{n/4}).$$
    This concludes the proof.
\end{proof}

\subsection{Newton's formula}

To prove the general case of Theorem~\ref{Prop_k_general}, we use a combinatorial argument.

Let $x_1$, $x_2$, $\cdots$ be an infinite collection of indeterminates. If a formal power series  $P(x_1, x_2, \cdots)$ with bounded degree is invariant under all finite permutations of the variables $x_1$, $x_2$, $\cdots$, we call it a \textit{symmetric function}. We define 
the $n$-th \textit{homogeneous symmetric function} $h_n=h_n(x_1, x_2, \cdots)$ by the following generating function
$$\sum_{n=0}^{\infty} h_n z^n=\prod_{i=1}^{\infty}\frac{1}{1-x_i z}.$$
Thus, 
$h_n$ is the sum of all possible monomials of degree $n$. The \textit{n-th elementary symmetric function} $e_n=e_n(x_1, x_2, \cdots)$ is defined by $$\sum_{n=0}^{\infty}e_n z^n=\prod_{i=1}^{\infty}(1+x_i z).$$ Precisely,  $e_n$ is the sum of all square-free monomials of degree $n$. Finally the $n$-th \textit{power symmetric function} 
$p_n=p_n(x_1, x_2, \cdots)$ is defined to be
$$p_n=x_1^n+x_2^n+\cdots.$$

The following result is due to Newton or Girard (see \cite[Chap.~1, (2.11)]{Mac}, or \cite[Th.~2.8]{Me-Re}).
\begin{lem}\label{lem-Newton}
	For any integer $k\geq 1$, we have
	\begin{equation*}
	kh_k=\sum_{\ell=1}^k h_{k-\ell}p_{\ell},
\qquad
ke_k=\sum_{\ell=1}^k (-1)^{\ell} e_{k-\ell} p_\ell.
\end{equation*}
\end{lem}

\subsection{Products of $k$ irreducible polynomials --- Induction step}

We will prove Theorem~\ref{Prop_k_general} by induction on $k$. First we use the combinatorial arguments from Lemma~\ref{lem-Newton} to obtain a relation between $\pi_{f_k}$ and $\pi_{f_{k-1}}$, the two relations are obtained by different calculations according to whether $f= \Omega$ or $\omega$.

\begin{lem}\label{lem recurence big Omega}
Let $M\in \mathbf{F}_q[t]$ of degree $d\geq1$, and $\chi$ be a non-trivial Dirichlet character modulo $M$. For any positive integer $k  \geq 2$, assume that for all $1\leq \ell\leq k-1$ there exists $A_{\Omega,\ell}>0$ such that one has
$\lvert\pi_{\Omega_{\ell}}(n,\chi)\rvert \leq A_{\Omega,\ell} \frac{d^{\ell}}{(\ell -1)!} \frac{q^{n/2}}{n}(\log n)^{\ell -1}$ for all~$n\geq 1$.
Then one has
\begin{align*}
\pi_{\Omega_k}(n, \chi)=\frac{1}{k}\sum_{n_{1} + n_{2} = n}&\pi_{\Omega_{k-1}}(n_{1},\chi)\pi_{\Omega_1}(n_{2},\chi)+O_k \left(\frac{q^{n/2}(\log n)^{k-2}}{n} \right),
\end{align*}
where the implicit constant depends on $k$ and is bounded by $$\frac{d^{k}}{k!}\sum_{\ell=2}^{k} ( 2 + \frac{\ell}{\log n})A_{\Omega,k-\ell}\frac{d^{-\ell} q^{-\ell/2+1}(\log n)^{2-\ell}(k-1)!}{(k-\ell -1)!}$$ for all~$n$.
\end{lem}

\begin{proof}
We study the function 
$$F_{\Omega_k}(u,\chi) = \sum_{n=1}^{\infty} \sum_{\substack{N \text{ monic} \\ \deg(N) = n \\ \Omega(N)=k}} \chi(N) u^{n} = \sum_{n=1}^{\infty}\pi_{\Omega_k}(n,\chi) u^{n}.$$
Adapting the idea of \cite{Meng2017}, we choose $x_P=\chi(P)u^{\deg P}$ for each irreducible polynomial $P$ . 
Using Lemma~\ref{lem-Newton},
we obtain
\begin{equation}\label{general-k-series}
F_{\Omega_k}(u,\chi) = \frac{1}{k}\sum_{\ell=1}^{k}F_{\Omega_{k-\ell}}(u,\chi)F_{\Omega_1}(u^{\ell},\chi^{\ell}),
\end{equation}
where we use the convention $F_{\Omega_0}(u,\chi) = 1$.
Comparing the coefficients of degree $n$, we see that the first term will give the main term and the other terms contribute to the error term.
For $2\leq \ell\leq k-1 $, using the trivial bound for $\pi_{\Omega_1}$ the coefficient of degree~$n$ of~$F_{\Omega_{k-\ell}}(u,\chi)F_{\Omega_1}(u^{\ell},\chi^{\ell})$ is indeed by hypothesis: 
\begin{align}\label{general-k-error}
\sum_{n_{1} + \ell n_{2} = n}&\pi_{\Omega_{k-\ell}}(n_{1},\chi)\pi_{\Omega_1}\left(n_{2},\chi^{\ell} \right) \\
&\leq A_{\Omega,k-\ell}   \frac{d^{k-\ell}}{(k-\ell-1)!}\sum_{n_{1} + \ell n_{2} = n}\frac{q^{n_{1}/2} q^{n_{2}} }{n_{1} n_{2}} (\log n_1)^{k-\ell-1}\nonumber\\
&\leq A_{\Omega,k-\ell} \frac{d^{k-\ell}}{(k-\ell -1)!} q^{n/2-\ell/2+1}(\log n)^{k-\ell -1} \sum_{n_1+\ell n_2=n}\frac{1}{n_1 n_2}  \nonumber\\
&\leq A_{\Omega,k-\ell}\frac{d^{k-\ell} q^{n/2-\ell/2+1}(\log n)^{k-\ell -1}}{(k-\ell -1)! n}(2 \log n + \ell) \nonumber. 
\end{align}
The coefficient of degree~$n$ of~$F_{\Omega_{0}}(u,\chi)F_{\Omega_1}(u^{k},\chi^{k})$ is non-zero only when $k\mid n$, and it is bounded by $\lvert \pi_{\Omega_1}(\tfrac{n}{k},\chi^{k})\rvert \ll \frac{k q^{\frac{n}{k}}}{n} \leq 2A_{\Omega,0} \frac{q^{n/2 - k/2 + 1}}{n}$, for a good choice of $A_{\Omega,0} >0$.
Then, by \eqref{general-k-series} and \eqref{general-k-error}, summing over $2\leq \ell\leq k$ we obtain Lemma~\ref{lem recurence big Omega}.
\end{proof}

\begin{lem}\label{lem recurence small omega}
Let $M\in \mathbf{F}_q[t]$ of degree $d\geq1$, and $\chi$ be a non-trivial Dirichlet character modulo $M$. For any positive integer $k \geq 2$, assume that for all $2\leq \ell\leq k-1$ there exists $A_{\omega, \ell}>0$ such that one has
$\lvert\pi_{\omega_{\ell}}(n,\chi)\rvert \leq A_{\omega,\ell} \frac{d^{\ell}}{(\ell -1)!} \frac{q^{n/2}}{n}(\log n)^{\ell -1}$
for all $n\geq 1$.
Then one has
\begin{align*}
    \pi_{\omega_k}(n,\chi) = \frac{1}{k} \sum_{n_1 + n_2 = n}\pi_{\omega_{k-1}}(n_1,\chi)\pi_{\omega_1}(n_2,\chi) 
    + O_{k}\left(\frac{q^{n/2}(\log n)^{k-2}}{n}  \right).
\end{align*}
where the implicit constant depends on $k$ and is bounded by $$2\frac{d^{k}}{k!} \sum_{\ell = 2}^{k}\sum_{j = \ell}^{n-1}A_{\omega,k-\ell}j\binom{j-1}{\ell-1}\frac{q^{1-\frac{j}{2}}d^{1-\ell}(\log n)^{2-\ell}(k-1)!}{(k-\ell -1)!} $$ for all~$n$.
\end{lem}

\begin{proof}
We study the function 
$$F_{\omega_k}(u,\chi) = \sum_{n=1}^{\infty} \sum_{\substack{N \text{ monic} \\ \deg(N) = n \\ \omega(N)=k}} \chi(N) u^{n} = \sum_{n=1}^{\infty}\pi_{\omega_k}(n,\chi) u^{n}.$$
Adapting the idea of \cite{Meng2017} for $x_{P} = \sum_{j\geq 1}\chi(P)^j u^{j\deg P}$, and 
using Lemma~\ref{lem-Newton},
we obtain
\begin{equation}\label{littleomega-general-k-series}
F_{\omega_k}(u,\chi) = \frac{1}{k}\sum_{\ell=1}^{k}(-1)^{\ell +1}F_{\omega_{k-\ell}}(u,\chi)\tilde{F}(u,\chi;\ell),
\end{equation}
where 
\begin{align*}
  \tilde{F}(u,\chi;\ell) = \sum\limits_{P\text{ irred.}}\( \sum\limits_{j\geq 1} \chi(P)^{j}u^{j\deg P} \)^{\ell}  
  = \sum\limits_{P\text{ irred.}}\sum\limits_{j\geq \ell}\binom{j-1}{\ell -1} \chi(P)^{j}u^{j\deg P}  .
\end{align*}
 Note that $\tilde{F}(u,\chi;1) = F_{\omega_1}(u,\chi)$, and we use the convention $F_{\omega_0}(u,\chi) = 1$.
Then we compare the coefficients of $u^n$ in \eqref{littleomega-general-k-series}, we show that the terms for $\ell \geq 2$ all contribute to the error term.
For $2\leq \ell \leq k-1$, the coefficient of degree $n$ of $F_{\omega_{k-\ell}}(u,\chi)\tilde{F}(u,\chi;\ell)$ is indeed by hypothesis: 
\begin{align*}
\sum_{j\geq \ell}\binom{j-1}{\ell -1}\sum_{n_{1} + j n_{2} = n}&\pi_{\omega_{k-\ell}}(n_{1},\chi)\pi_{\Omega_1}\left(n_{2},\chi^{j} \right) \\
&\leq   A_{\omega,k-\ell}\sum_{j = \ell}^{n-1}\binom{j-1}{\ell -1} \frac{d^{k-\ell +1}}{(k-\ell-1)!}\sum_{n_{1} + j n_{2} = n}\frac{q^{n_{1}/2} q^{n_{2}} }{n_{1} n_{2}} (\log n_1)^{k-\ell-1}\nonumber\\
&\leq A_{\omega,k-\ell} \frac{d^{k-\ell+1}(\log n)^{k-\ell-1} q^{n/2}}{(k-\ell-1)!}\sum_{j = \ell}^{n-1}\binom{j-1}{\ell -1}\sum_{n_1+ j n_2=n} \frac{q^{-(j/2-1)n_2}}{n_1 n_2} \nonumber\\
&\leq 2 A_{\omega,k-\ell}\frac{d^{k}}{(k-\ell -1)!} \frac{q^{n/2}(\log n)^{k-\ell}}{n}  \sum_{j = \ell}^{n-1} j\binom{j-1}{\ell -1} q^{1-j/2} d^{1-\ell},\nonumber
\end{align*}
The coefficient of degree $n$ of $F_{\omega_{0}}(u,\chi)\tilde{F}(u,\chi;k)$ is bounded by
$$\sum_{\substack{k\leq j \leq n-1 \\ j\mid n  }}\binom{j-1}{k-1}\frac{j q^{\frac{n}{j}}}{n} \leq 2A_{\omega,0} \frac{q^{n/2}}{n}  \sum_{j = k}^{n-1} j\binom{j-1}{k -1} q^{1-j/2},$$
for a good choice of $A_{\omega,0}>0$.
Summing over $2\leq \ell\leq k$ we obtain Lemma~\ref{lem recurence small omega}.
\end{proof}

In order to avoid some confusions with complete sum over all zeros, in the following we use $\sum'$ to represent the sum over non-real zeros of the $L$-function. We also assume all the multiplicities and zeros depend on $\chi$ in this section.  

For $f = \Omega$ or $\omega$, $\ell \geq 1$ and $\chi \bmod M$ (with the convention $0!=1$), we denote
\begin{align*}
Z_{\ell}(n,\chi)& = \frac{(-1)^{\ell}}{(\ell-1)!} \sideset{}{'}\sum_{1\leq j\leq d_{\chi}} m_j^{\ell} \frac{\alpha_j^n(\chi) (\log n)^{\ell-1}}{n},\\
B_{f_\ell}(n,\chi) &= \frac{(-1)^{\ell}}{(\ell-1)!} \( \(m_+ -\epsilon_f\frac{\delta(\chi^2)}{2}\)^{\ell} +(-1)^n \(m_- -\epsilon_f \frac{\delta(\chi^2)}{2}\)^\ell\)\frac{q^{n/2}(\log n)^{\ell-1}}{n},
\end{align*}
where $\epsilon_{\Omega} = -1$ and $\epsilon_{\omega} = 1$.
With these notations, we rewrite the formula in Theorem~\ref{Prop_k_general} and Proposition~\ref{Prop_k=1} in the following form: there exists positive constants $C_{f,\ell}$ such that
\begin{align}\label{formula Th deg=n}
E_{f_\ell}(n,\chi) := \lvert\pi_{f_\ell}(n,\chi) - Z_{\ell}(n,\chi)  - B_{f_\ell}(n,\chi)\rvert  \leq \begin{cases}   C_{f,\ell}\frac{d^{\ell}}{(\ell-1)!}\frac{1}{\gamma(\chi)}\frac{q^{n/2} (\log n)^{\ell-2}}{n}& \text{ for } \ell \geq 2\\
C_{f,1}d \frac{q^{n/3}}{n}  & \text{ for } \ell =1,
\end{cases}
\end{align}
where for $2\leq \ell =o((\log n)^{\frac12})$, and if $f = \omega$ then $q>3$, we need to show that $C_{f,\ell} \leq C\ell(\ell -1)$ with $C$ an absolute constant.
By Lemma~\ref{lem recurence big Omega} (resp. \ref{lem recurence small omega}), it suffices to study the coefficient of $u^n$ in $F_{f_{k-1}}(u,\chi)F_{f_1}(u,\chi)$, that is:
\begin{align}\label{general-k-main}
\sum_{n_{1} + n_{2} = n}&\pi_{f_{k-1}}(n_{1},\chi)\pi_{f_1}(n_{2},\chi)\nonumber \\
=& \sum_{n_{1} + n_{2} = n} \big\{ Z_{k-1}(n_{1},\chi)+B_{f_{k-1}}(n_{1},\chi)+E_{f_{k-1}}(n_{1},\chi) \big\}\big\{Z_1(n_{2},\chi)+B_{f_1}(n_{2},\chi)+E_{f_1}(n_{2},\chi)\big\}\nonumber\\
=&\sum_{n_{1} + n_{2} = n}Z_{k-1}(n_1, \chi)Z_1(n_2, \chi)+\sum_{n_{1} + n_{2} = n}B_{f_{k-1}}(n_1, \chi)B_{f_1}(n_2, \chi)\nonumber\\
&\quad + \sum_{n_{1} + n_{2} = n}\big\{ Z_{k-1}(n_1, \chi)B_{f_1}(n_2, \chi)+B_{f_{k-1}}(n_1, \chi)Z_1(n_2, \chi)\big\}\nonumber \\
&\quad +\sum_{n_{1} + n_{2} = n} \big\{ \(Z_{k-1}(n_{1},\chi)+B_{f_{k-1}}(n_{1},\chi)\) E_{f_1}(n_2, \chi) \big\}
+\sum_{n_{1} + n_{2} = n}   E_{f_{k-1}}(n_{1},\chi) \pi_{f_1}(n_2,\chi).
\end{align}
We will now study each of these sums separately.

\subsection{Bounds for certain exponential sums}

We first give a bound for certain exponential sums that appear several times in the proof of Lemmas~\ref{Lem_general_k_Zeros}--\ref{Lem-mixed-term}.
The following result follows from partial summation. 
\begin{lem}\label{Lem_Abel}
Let $f$ be a differentiable function on $[1,+\infty)$ such that $f'(x) \in L^{1}[1,\infty)$.
Then for every $\theta \in (-\frac{\pi}{2}, \frac{\pi}{2}]$, $\theta \neq 0$,
one has
\begin{align*}
\sum_{n=1}^{N} e^{i\theta n}f(n)  =  O\left(\frac{ \lVert f' \rVert_{L^{1}} + \lVert f \rVert_{\infty}}{\lvert\theta\rvert}\right)
\end{align*}
as $N \rightarrow +\infty$, with an absolute implicit constant.
\end{lem}

\begin{proof}
As $e^{i\theta} \neq 1$, one has
\begin{align*}
H(x) := \sum_{n\leq x} e^{i\theta n} =  \frac{e^{i\theta}- e^{i\theta ([x]+1)}}{1-e^{i\theta}} = O\(\frac{1}{\lvert\theta\rvert}\).
\end{align*}
So applying Abel's identity, one has
\begin{align*}
\sum_{n=1}^{N} e^{i\theta n}f(n)  = H(N)f(N) + \int_{1}^{N} H(t)f'(t) \diff t 
= O\left(\frac{\lvert f(N)\rvert}{\lvert\theta\rvert}\right) +  O\left(\frac{1}{\lvert\theta\rvert}\int_{1}^{N} \lvert f'(t) \rvert \diff t\right).
\end{align*}
\end{proof}

\subsection{Sum over non-real zeros}
\begin{lem}\label{Lem_general_k_Zeros}
For any $k \geq 2$, one has
\begin{multline*}
   \sum_{n_{1} + n_{2} = n}Z_{k-1}(n_{1},\chi)Z_{1}(n_{2},\chi) \\
= \frac{(-1)^{k}k}{(k-1)!} \left\{ \sideset{}{'}\sum_{j=1}^{d_{\chi}} m_{j}^{k} \frac{\alpha_{j}(\chi)^{n}(\log n)^{k-1}}{n} + O\left( d^k \left( k + \frac{1}{\gamma(\chi)}\right)\frac{q^{n/2}(\log n)^{k-2}}{n}\right) \right\}, 
\end{multline*}
where the implicit constant is absolute.
\end{lem}

\begin{proof}
We separate the sum in a diagonal term and off-diagonal term:
 \begin{align*}
\sideset{}{'}\sum_{j_{1}=1}^{d_{\chi}} \sideset{}{'}\sum_{j_{2}=1}^{d_{\chi}} \sum_{n_{1}+ n_{2} = n}\frac{(-1)^k}{(k-2)!}m_{j_{1}}^{k-1}m_{j_{2}}\frac{\alpha_{j_{1}}(\chi)^{n_{1}}\alpha_{j_{2}}(\chi)^{n_{2}}(\log n_{1})^{k-2}}{n_{1}n_{2}} = \Sigma_{1} + \Sigma_{2},
\end{align*}
where
\begin{align*}
\Sigma_{1} =\frac{(-1)^k}{(k-2)!}\sideset{}{'}\sum_{j=1}^{d_{\chi}}\sum_{n_{1}+ n_{2} = n} m_{j}^{k}\frac{\alpha_{j}(\chi)^{n_{1} + n_{2}}(\log n_{1})^{k-2}}{n_{1}n_{2}}, 
\end{align*}
and 
\begin{align*}
\Sigma_{2} = \frac{(-1)^k}{(k-2)!}\sideset{}{'}\sum_{j_{1}\neq j_2} \sum_{n_{1}+ n_{2} = n}m_{j_{1}}^{k-1}m_{j_{2}}\frac{\alpha_{j_{1}}(\chi)^{n_{1}}\alpha_{j_{2}}(\chi)^{n_{2}}(\log n_{1})^{k-2}}{n_{1}n_{2}}. 
\end{align*}

The diagonal term gives the main term, for $1\leq j \leq d$ one has
\begin{align}\label{main-zero}
\sum_{n_{1}+ n_{2} = n} m_{j}^{k}\frac{\alpha_{j}(\chi)^{n_{1} + n_{2}}(\log n_{1})^{k-2}}{n_{1}n_{2}}= m_{j}^{k}\frac{\alpha_{j}(\chi)^{n}}{n} \sum_{n_{1} + n_{2} = n}\left(\frac{(\log n_{1})^{k-2}}{n_{1}} + \frac{(\log n_{1})^{k-2}}{n_{2}} \right).
\end{align}
By partial summation, we have
\begin{equation}\label{sum1-1}
\sum_{n_1+n_2=n} \frac{(\log n_{1})^{k-2}}{n_{1}}=\sum_{n_1=1}^{n-1} \frac{(\log n_{1})^{k-2}}{n_{1}}=\frac{1}{k-1} \left( (\log n)^{k-1}+O\left( k(\log n)^{k-2}\right) \right).
\end{equation}
For the second sum in \eqref{main-zero}, we have
\begin{align}\label{sum1-2-0}
\sum_{n_1+n_2=n}\frac{(\log n_{1})^{k-2}}{n_{2}}&=\sum_{1\leq n_2\leq n/2} \frac{(\log(n-n_2))^{k-2}}{n_{2}}+\sum_{n/2<n_2<n}\frac{(\log(n-n_2))^{k-2}}{n_{2}} \nonumber \\
&=\sum_{1\leq n_2\leq n/2} \frac{(\log n+\log (1-n_2/n))^{k-2}}{n_{2}}+O\(\frac{n}{2}\cdot \frac{(\log n)^{k-2}}{n} \),
\end{align}
for $1\leq n_2\leq n/2$, $|\log(1-n_2/n)|<1$, thus
\begin{align}\label{sum1-2}
\sum_{n_1+n_2=n}\frac{(\log n_{1})^{k-2}}{n_{2}}&=\sum_{1\leq n_2\leq n/2}\( \frac{(\log n)^{k-2} }{n_2}+ \frac{O(k(\log n)^{k-3})}{n_2}\)+O\((\log n)^{k-2}\) \nonumber\\
&=(\log n)^{k-1}+O\(k(\log n)^{k-2}\).
\end{align}
Inserting \eqref{sum1-1} and \eqref{sum1-2} into \eqref{main-zero},  we get
\begin{equation}\label{sum1-3}
\sum_{n_{1}+ n_{2} = n} m_{j}^{k}\frac{\alpha_{j}(\chi)^{n_{1} + n_{2}}(\log n_{1})^{k-2}}{n_{1}n_{2}}=\frac{k m_j^k}{k-1}\left( \frac{\alpha_j(\chi)^{n}(\log n)^{k-1}}{n}+O\(k\frac{q^{n/2}(\log n)^{k-2}}{n}\) \right).
\end{equation}
Thus, 
\begin{align*}
\Sigma_1=\frac{(-1)^k k}{(k-1)!}\left\{\sideprime\sum_{j=1}^{d_{\chi}} m_j^k \frac{\alpha_j(\chi)^{n}(\log n)^{k-1}}{n}+O\(d^k k\frac{q^{n/2}(\log n)^{k-2}}{n}\)\right\}.
\end{align*}

For $\alpha_{j_{1}} \neq \alpha_{j_{2}}$, one has
\begin{align}\label{sum-2-1}
\sum_{n_{1}+ n_{2} = n}\frac{\alpha_{j_{1}}(\chi)^{n_{1}}\alpha_{j_{2}}(\chi)^{n_{2}}(\log n_{1})^{k-2}}{n_{1}n_{2}} 
= &\frac{\alpha_{j_{2}}(\chi)^{n}}{n}\sum_{n_{1} =1}^{n-1}\frac{\left(\alpha_{j_{1}}(\chi)/\alpha_{j_{2}}(\chi)\right)^{n_{1}}(\log n_{1})^{k-2}}{n_{1}}\nonumber\\& + 
\frac{\alpha_{j_{1}}(\chi)^{n}}{n}\sum_{n_{2} =1}^{n-1}\frac{\left(\alpha_{j_{2}}(\chi)/\alpha_{j_{1}}(\chi)\right)^{n_{2}}(\log (n-n_{2}))^{k-2}}{n_{2}},
\end{align}
where $\lvert \alpha_{j_{1}}(\chi)/\alpha_{j_{2}}(\chi) \rvert =1$, and $\alpha_{j_{1}}(\chi)/\alpha_{j_{2}}(\chi) \neq 1$.
We apply Lemma~\ref{Lem_Abel}  with $f(x) = \frac{(\log x)^{k-2}}{x}$ to the first sum to deduce that this sum is $O\left( (\log n)^{k-2} \lvert \gamma_{j_1}  - \gamma_{j_2} \rvert^{-1} \right)$.
The second term can be separated at $\frac{n}{2}$ as in \eqref{sum1-2-0}, it yields
\begin{equation*}
\sum_{1\leq n_2\leq n/2}\(\frac{\left(\alpha_{j_{2}}(\chi)/\alpha_{j_{1}}(\chi)\right)^{n_{2}} (\log n)^{k-2} }{n_2}+ \frac{O(k(\log n)^{k-3})}{n_2}\)+O\((\log n)^{k-2}\).
\end{equation*}
Then we apply Lemma~\ref{Lem_Abel} with $f(x) = \frac{1}{x}$ to the first term above.
In the end we obtain
\begin{align*}
\Sigma_{2} = O\left( d^k \frac{k\left(k + \frac{1}{\gamma(\chi)}\right)}{(k-1)!}\frac{q^{n/2}(\log n)^{k-2}}{n}\right).
\end{align*}

The proof of Lemma~\ref{Lem_general_k_Zeros} is complete.
\end{proof}

\subsection{Bias term}
\begin{lem}\label{Lem_general_k_Bias} For $f=\Omega$ or $\omega$, and for any $k\geq 2$, we have
\begin{multline*}
    \sum_{n_{1} + n_{2} = n} B_{f_{k-1}}(n_1, \chi)B_{f_1}(n_2, \chi)
=  \frac{(-1)^k k}{(k-1)!}\frac{q^{n/2}(\log n)^{k-1}}{n}\Bigg\{ \(m_+ - \epsilon_f\frac{\delta(\chi^2)}{2}\)^k \\ +(-1)^n \(m_- -\epsilon_f\frac{\delta(\chi^2)}{2}\)^k     +O \(d^k k (\log n)^{-1}  \) \Bigg\},
\end{multline*}
where $\epsilon_{\Omega} = -1$, $\epsilon_{\omega} =1$ and the implicit constant is absolute.
\end{lem}

\begin{proof}
We write the sum as sum of four parts, 
\begin{align*}
\sum_{n_{1} + n_{2} = n} B_{f_{k-1}}(n_1, \chi)&B_{f_1}(n_2, \chi) \nonumber \\
=\frac{(-1)^k q^{n/2}}{(k-2)!} \sum_{n_1+n_2=n}\Bigg\{&\( m_+ -\epsilon_f \frac{\delta\(\chi^2\)}{2}\)^{k-1}
\( m_+ -\epsilon_f \frac{\delta\(\chi^2\)}{2} \) \nonumber\\
&+ 
\( m_+ -\epsilon_f \frac{\delta\(\chi^2\)}{2}\)^{k-1}
(-1)^{n_2}\( m_{-} -\epsilon_f \frac{\delta\(\chi^2\)}{2}\) \nonumber\\
&+ 
(-1)^{n_1}\( m_- -\epsilon_f \frac{\delta\(\chi^2\)}{2}\)^{k-1}
\( m_{+} -\epsilon_f \frac{\delta\(\chi^2\)}{2}\) \nonumber\\
&+ 
(-1)^{n_1}\( m_- -\epsilon_f \frac{\delta\(\chi^2\)}{2}\)^{k-1}
(-1)^{n_2}\( m_{-} -\epsilon_f \frac{\delta\(\chi^2\)}{2}\)
\Bigg\} \frac{ (\log n_1)^{k-2}}{n_1 n_2}\nonumber\\
=&: \frac{(-1)^k q^{n/2}}{(k-2)!} \left\{ S_1+S_2+S_3+S_4 \right\}.
\end{align*}
First, we see that $S_1$ and $S_4$ should give the main term, and we expect $S_2$ and $S_3$ to be in the error term. 
Using \eqref{sum1-1} and \eqref{sum1-2}, we have
\begin{align*}
\sum_{n_1+n_2=n}  \frac{(\log n_1)^{k-2}}{n_1 n_2} = \frac{k}{k-1}\frac{(\log n)^{k-1}}{n}  + O\( k\frac{(\log n)^{k-2}}{n} \).
\end{align*}
Thus
\begin{align*}
S_1 &= \frac{k}{k-1}\( m_+ -\epsilon_f \frac{\delta\(\chi^2\)}{2}\)^{k}\frac{(\log n)^{k-1}}{n}  + O\( k\( m_+ -\epsilon_f \frac{1}{2}\)^{k} \frac{(\log n)^{k-2}}{n} \), \\
S_4 &= (-1)^n\frac{k}{k-1}\( m_- -\epsilon_f \frac{\delta\(\chi^2\)}{2}\)^{k}\frac{(\log n)^{k-1}}{n}  + O\( k\( m_- -\epsilon_f \frac{1}{2}\)^{k} \frac{(\log n)^{k-2}}{n} \).
\end{align*}
Similar to \eqref{sum-2-1}, we have
\begin{align*}
\sum_{n_1+n_2=n} (-1)^{n_1} \frac{(\log n_1)^{k-2}}{n_1 n_2} =  O\( \left(k + \frac{1}{\pi}\right)\frac{(\log n)^{k-2}}{n} \).
\end{align*}
Thus 
\begin{align*}
    S_2 + S_3 =  O\( d^k k\frac{(\log n)^{k-2}}{n} \)
\end{align*}
Combining $S_1$, $S_2$, $S_3$ and $S_4$ we obtain Lemma~\ref{Lem_general_k_Bias}.
\end{proof}

\subsection{Other error terms}

\begin{lem}\label{lem_general_k_mixed-bias-zero}
For $f=\Omega$ or $\omega$ and for any $k\geq 2$, one has
\begin{align*}
\sum_{n_{1} + n_{2} = n}\big\{ Z_{k-1}(n_1, \chi)B_{f_1}(n_2, \chi)+B_{f_{k-1}}(n_1, \chi)Z_1(n_2, \chi)\big\} =  O\( d^k \frac{k \left( k + \frac{1}{\gamma(\chi)}\right)}{(k-1)!}\frac{q^{n/2} (\log n)^{k-2}}{n}\),
\end{align*}
where the implicit constant is absolute.
\end{lem}

\begin{proof}
Let $\alpha_{j}$ be a non-real inverse zero of the $L$-function, one has
\begin{multline}\label{sum_zero_bias_1-2}
\sum_{n_1+n_2=n}m_j^{k-1} \(\(m_{+} -\epsilon_f \frac{\delta(\chi^2)}{2}\) +(-1)^{n_2}\(m_{-} -\epsilon_f \frac{\delta(\chi^2)}{2}\)\)\frac{\alpha_{j}^{n_{1}}(\log n_1)^{k-2}}{n_{1}} \frac{q^{n_{2}/2}}{n_{2}} \\
=  O \( m_j^{k-1} ( \max\(m_{+},m_{-}\) + \tfrac{1}{2}) \left( k  + \frac{1}{\min(\lvert \gamma_j\rvert, \lvert\pi - \gamma_j\rvert) }\right)\frac{q^{n/2} (\log n)^{k-2}}{n}\),
\end{multline}
this follows from the same idea as for \eqref{sum-2-1}. 
We sum \eqref{sum_zero_bias_1-2} over the zeros to obtain
\begin{align*}
\sum_{n_{1} + n_{2} = n} Z_{k-1}(n_1, \chi)B_{f_1}(n_2, \chi) 
&=  O\( d^k \frac{k \left( k + \frac{1}{\gamma(\chi)}\right)}{(k-1)!}\frac{q^{n/2} (\log n)^{k-2}}{n}\).
\end{align*}
The proof is similar for the other term.
\end{proof}

\begin{lem}\label{Lem-mixed-term}
For $f = \Omega$ or $\omega$ and for any $k\geq 2$,	one has
	\begin{align*}
	\sum_{n_{1} + n_{2} = n} & \(Z_{k-1}(n_{1},\chi)+B_{f_{k-1}}(n_{1},\chi)\) E_{f_1}(n_2, \chi)   = O\( d^k \frac{k}{(k-1)!} \frac{q^{n/2} (\log n)^{k-2}}{n} \).
	\end{align*}
\end{lem}

\begin{proof}
We use the following bound, for $k-1 \geq 1$:  
 \begin{align*}
\lvert Z_{k-1}(n,\chi) + B_{f_{k-1}}(n,\chi)\rvert  \leq  d^{k-1} \frac{1}{(k-2)!} \frac{q^{n/2}(\log n)^{k-2}}{n},
\end{align*}
where the implicit constant is absolute.
In particular the term evaluated in Lemma~\ref{Lem-mixed-term} satisfies
\begin{align*}
\sum_{n_{1} + n_{2} = n}& \(Z_{k-1}(n_{1},\chi)+B_{f_{k-1}}(n_{1},\chi)\) E_{f_1}(n_2, \chi) \\
&= \sum_{n_{1} + n_{2} = n}  O\( d^{k} \frac{1}{(k-2)!} \frac{q^{n_1/2}q^{n_2/3}(\log n_1)^{k-2}}{n_1}\)  \\
&= O \(  d^{k} \frac{q^{n/2}(\log n)^{k-2}}{(k-2)!} \( \frac{1}{n}\sum_{n_{2} \leq n/2} q^{-{n_2}/6}  + q^{-n/6}\sum_{n_{1} \leq n/2}\frac{q^{n_{1}/6}}{n_1} \) \) \\
&= O \(\frac{d^k}{(k-2)!} \frac{q^{\frac{n}{2}}(\log n)^{k-2}}{n} \),
\end{align*}
with an absolute implicit constant.
This concludes the proof.
\end{proof}

\subsection{Proof of Theorem~\ref{Prop_k_general}}

We now have all the ingredients to finish the proof of Theorem~\ref{Prop_k_general}.

\begin{proof}[Proof of Theorem~\ref{Prop_k_general}]
By induction on $k$, the base case is Proposition~\ref{Prop_k=1} ($k=1$). Now suppose for any $2\leq \ell\leq k-1$, we have
\begin{align}\label{Eq bound Error}
E_{f_\ell}(n,\chi)  \leq C_{f,\ell} \frac{d^{\ell}}{(\ell-1)!}\frac{1}{\gamma(\chi)}\frac{q^{n/2} (\log n)^{\ell-2}}{n},
\end{align}
where $C_{f,\ell}\leq C\ell(\ell-1)$ as stated in \eqref{formula Th deg=n}.
In particular, the condition of Lemma~\ref{lem recurence big Omega} (resp. Lemma~\ref{lem recurence small omega}) is satisfied for $k$, one has
\begin{align}
    \lvert \pi_{f_\ell}(n,\chi) \rvert &\leq \lvert Z_{\ell}(n,\chi) + B_{f_{\ell}}(n,\chi)\rvert + C_{f,\ell} \frac{d^{\ell}}{(\ell-1)!}\frac{1}{\gamma(\chi)}\frac{q^{n/2} (\log n)^{\ell-2}}{n}\nonumber \\
    &\leq \Big(1 + \frac{C_{f,\ell}}{\gamma(\chi)\log n}\Big) \frac{d^{\ell}}{(\ell -1)!} \frac{q^{n/2}}{n}(\log n)^{\ell -1},
    \label{Eq bound on A Omega}
\end{align}
for $1\leq \ell \leq k-1$, for all $n\geq 2$.
Thus, take $A_{\Omega, \ell}=1 + \frac{C_{\Omega,\ell}}{\gamma(\chi)\log n}$ in Lemma~\ref{lem recurence big Omega}, and evaluate each sum in Equation~\eqref{general-k-main}  thanks to Lemmas~\ref{Lem_general_k_Zeros}--\ref{Lem-mixed-term}, this yields
\begin{align}\label{Eq induction Omega}
   \frac{n}{q^{n/2}(\log n)^{k-2}}\lvert E_{\Omega_{k}}(n,\chi) \rvert \leq& \frac{d^{k}}{k!}\sum_{\ell=2}^{k} \Big(1 + \frac{C_{\Omega,k-\ell}}{\gamma(\chi)\log n}\Big)\Big(2 + \frac{\ell}{\log n}\Big)\frac{d^{-\ell} q^{-\ell/2+1}(\log n)^{2-\ell}(k-1)!}{(k-\ell -1)!} \nonumber\\ &+ \frac{C_0}{2} d^k \frac{k + \frac{1}{\gamma(\chi)}}{(k-1)!}  + \frac{1}{k} \frac{n}{q^{n/2}(\log n)^{k-2}}\sum_{n_1 + n_2 = n}\lvert E_{\Omega_{k-1}}(n_1,\chi) \pi_{f_1}(n_2,\chi) \rvert,
\end{align}
where $C_0$ is an absolute constant. %and we may assume $C_0 \geq 6\pi$.

In the case $k=2$, we get
\begin{align*}
   \frac{n}{q^{n/2}}\lvert E_{\Omega_{2}}(n,\chi) \rvert \ll  \frac{d^2}{\gamma(\chi)}  + \frac{n}{q^{n/2}} \sum_{n_1 + n_2 = n} d\frac{q^{n_1 /3}}{n_1} d\frac{q^{n_2/2}}{n_2} 
   \ll  \frac{d^2}{\gamma(\chi)}
\end{align*}
which is the expected bound.
For $k\geq 3$, using the bound~\eqref{Eq bound Error}, we have
\begin{align}    \label{Eq applying induction Omega}
    \sum_{n_1 + n_2 = n}&\lvert E_{\Omega_{k-1}}(n_1,\chi) \pi_{f_1}(n_2,\chi) \rvert\\ &\leq \sum_{n_1 + n_2 = n}C_{\Omega,k-1}\frac{d^{k-1}}{(k-2)!}\frac{1}{\gamma(\chi)}\frac{q^{n_1/2}(\log n_1)^{k-3}}{n_1} \Big(1 + C_{\Omega,1}q^{-n_2/6}\Big) d \frac{q^{n_2/2}}{n_2}\nonumber \\
    &\leq C_{\Omega,k-1}\frac{d^k}{(k-2)!}\frac{1}{\gamma(\chi)}\frac{q^{n/2}}{n}\sum_{n_1 + n_2 = n} \left(\frac1{n_1} + \frac{1}{n_2}\right)(\log n_1)^{k-3}\Big(1 + C_{\Omega,1}q^{-n_2/6}\Big)\nonumber  \\
    &\leq C_{\Omega,k-1}\frac{d^k}{(k-2)!}\frac{1}{\gamma(\chi)}\frac{q^{n/2}}{n}\left( \frac{k-1}{k-2}(\log n)^{k-2} + O(k(\log n)^{k-3} ) \right),\nonumber
\end{align}
which together with the bound~\eqref{Eq induction Omega}, proves the existence of $C_{\Omega,k}$ satisfying
\begin{align*}
E_{\Omega_k}(n,\chi)  \leq C_{\Omega,k} \frac{d^{k}}{(k-1)!}\frac{1}{\gamma(\chi)}\frac{q^{n/2} (\log n)^{k-2}}{n}.
\end{align*}
Now, when $k = o((\log n)^{1/2})$, by the induction hypothesis~\eqref{Eq bound Error}, one has $C_{\Omega,\ell} \leq  C\ell(\ell -1) =  o(\log n)$ for $2\leq \ell\leq k-1$ and some absolute constant $C$. In the following, we show how to choose $C$ and close the induction. We simplify the bounds~\eqref{Eq induction Omega} and \eqref{Eq applying induction Omega} to obtain
\begin{align*}
    C_{\Omega,k} \leq& \frac{2(k-1)(k-2)}{k}\sum_{\ell=2}^{k} (\gamma(\chi) + o(1))d^{-\ell} q^{-\ell/2+1} + \frac{C_0}{2}k + C_{\Omega,k-1}\frac{k-1}{k}\left( \frac{k-1}{k-2} + o(k^{-1}) \right) \\
    \leq& C\left(\frac{(k-1)(k-2)}{2k} + \frac{k}{2}  + \frac{(k-1)^3}{k} + o(k) \right)  \leq C k(k-1),
\end{align*}
if we choose $C\geq \max \lbrace C_0, 6\pi \rbrace \geq 4\gamma(\chi)\sum_{\ell=2}^{k} d^{-\ell} q^{-\ell/2+1}$ and for $k$ large enough (say $k\geq K$ finite). In the end, choose $C\geq \max\{ \frac{C_{\Omega, 2}}{2}, \cdots, \frac{C_{\Omega, K}}{K(K-1)}, C_0,6\pi \}$, we deduce that $C_{\Omega, k}\leq Ck(k-1)$ for all $2\leq k=o((\log n)^{1/2})$. This closes the induction step for $C_{\Omega,k}$.

The proof works similarly for $C_{\omega,k}$ using Lemma~\ref{lem recurence small omega}. For $k\geq 2$, we have
\begin{align}\label{Eq induction omega}
   \frac{n}{q^{n/2}(\log n)^{k-2}}\lvert E_{\omega_{k}}(n,\chi) \rvert \leq& 2\frac{d^{k}}{k!} \sum_{\ell = 2}^{k}\sum_{j = \ell}^{n-1}\left(1 + \frac{C_{\omega,k-\ell}}{\gamma(\chi)\log n}\right)j\binom{j-1}{\ell-1}\frac{q^{1-\frac{j}{2}}d^{1-\ell}(\log n)^{2-\ell}(k-1)!}{(k-\ell -1)!}\nonumber
   \\ &+ \frac{C_0}{2} d^k \frac{k + \frac{1}{\gamma(\chi)}}{(k-1)!}  + \frac{1}{k} \frac{n}{q^{n/2}(\log n)^{k-2}}\sum_{n_1 + n_2 = n}\lvert E_{\omega_{k-1}}(n_1,\chi) \pi_{f_1}(n_2,\chi) \rvert.
\end{align}
The last term is handled as in~\eqref{Eq applying induction Omega}. The first term is bounded independently of $n$ (but a priori not independently of $q$ if $q = 3$) by observing that the series
$$ \sum_{j \geq \ell} \frac{j!}{(j-\ell)!} q^{-\frac{j}{2}}   = \frac{\sqrt{q}}{\sqrt{q}-1} \ell ! (\sqrt{q} - 1)^{-\ell}$$
is convergent.
Up to increasing the constant to include the case $q=3$, this proves the existence of $C_{\omega,k}$ satisfying 
\begin{align*}
E_{\omega_k}(n,\chi)  \leq C_{\omega,k} \frac{d^{k}}{(k-1)!}\frac{1}{\gamma(\chi)}\frac{q^{n/2} (\log n)^{k-2}}{n}.
\end{align*}
Now, assuming $q \geq 5$, one has
\begin{align*}
  \sum_{\ell = 2}^{k}\sum_{j = \ell}^{n-1}j\binom{j-1}{\ell-1}q^{1-\frac{j}{2}}d^{1-\ell}  &\leq 2dq\sum_{\ell = 2}^{k}\ell (d(\sqrt{q} -1))^{-\ell}. 
\end{align*}
The series is convergent and can be bounded independently of $q$ and $d$, we may choose $C \geq \max\{ C_0, 8\gamma(\chi)dq\sum_{\ell \geq 2}\ell (d(\sqrt{q} -1))^{-\ell}\}$.
Thus, for $q\geq 5$, $k = o((\log n)^{1/2})$, using the induction hypothesis $C_{\omega,\ell} \leq C\ell(\ell -1)$ for $2\leq \ell\leq k-1$,  \eqref{Eq induction omega} becomes
\begin{align*}
   C_{\omega,k}\leq  C\left(\frac{(k-1)(k-2)}{2k} + \frac{k}{2}  + \frac{(k-1)^3}{k} + o(k) \right). \end{align*}
By the same argument as $C_{\Omega, k}$, we conclude that $C_{\omega, k}\leq C k(k-1)$ for some absolute constant $C$.
\end{proof}

\section{Counting polynomials of degree $\leq n$ with $k$ irreducible factors in congruence classes}\label{Sec_proof_deg<X}

The asymptotic formula in Theorem~\ref{Th_Difference_k_general_deg<X}  is obtained as a corollary of Theorem~\ref{Prop_k_general}, by summing over the characters and over the degree of the polynomials.

For $A \subset (\mathbf{F}_q[t]/(M))^{*}$, for $f = \Omega$ or $\omega$ and for any integers $n,k \geq 1$,
we define the function
\begin{equation*}\pi_{f_k}(n; M, A) = \lvert\lbrace N \in\mathbf{F}_{q}[t] : N \text{ monic, } \deg{N} = n,~ \Omega(N)=k,~ N \bmod M \in A  \rbrace\rvert,\end{equation*}
so that
\begin{align*}
\Delta_{f_k}(X; M, A, B)
 = \frac{X (k-1)!}{q^{X/2}(\log X)^{k-1}}\sum_{n\leq X}\left( \frac{1}{\lvert A\rvert }\pi_{f_k}(n; M, A) - \frac{1}{\lvert B \rvert}\pi_{f_k}(n; M, B)\right).
\end{align*}

Before we give the proof of Theorem \ref{Th_Difference_k_general_deg<X}, let us prove the following preliminary lemma. 
\begin{lem}\label{Lem_SumOverN}
Let $k\geq 0$ be an integer.
For any complex number $\alpha$ with $\lvert\alpha\rvert \geq \sqrt{2}$, as $X\rightarrow\infty$ we have that
\begin{align*}
\frac{X}{\alpha^{X}(\log X)^{k}}\sum_{n=1}^{X}\frac{\alpha^n(\log n)^{k}}{n}  = \frac{\alpha}{\alpha -1} + O\( \frac{1}{\lvert\alpha\rvert^{X}}+\frac{1 + \frac{k}{\log X}}{X\log X} \).
\end{align*}
\end{lem}

\begin{proof}
The proof is adapted from \cite[Lem.~2.2]{Cha2008}.
Applying Abel identity yields
\begin{align*}
\sum_{n=1}^{X} \frac{\alpha^n(\log n)^{k}}{n}  &= \frac{\alpha^{X+1} -{\alpha}}{\alpha -1}\frac{(\log X)^k}{X} + \int_{1}^{X} \frac{\alpha^{[t]+1} -{\alpha}}{\alpha -1} \frac{(k-1)(\log t)^{k-2} - (\log t)^{k-1}}{t^2} \diff t \\
&= \frac{\alpha}{\alpha -1}(\alpha^{X} +O(1)) \frac{(\log X)^k}{X} 
+  O\( \(k(\log X)^{k-2} + (\log X)^{k-1}\)\int_{1}^{X} \frac{ \lvert \alpha\rvert^{t}}{t^2} \diff t \).
\end{align*}
Cha proved that $\int_{1}^{X} \frac{ \lvert \alpha\rvert^{t}}{t^2} \diff t = O\( \frac{\lvert \alpha\rvert^X}{X^2}\)$ via integration by parts. This concludes the proof.
\end{proof}

\begin{proof}[Proof of Theorem~\ref{Th_Difference_k_general_deg<X}]
Let us first sum over the characters.
By orthogonality of characters, for every $A \subset (\mathbf{F}_{q}[t]/(M))^*$, one has
\begin{equation*}\pi_{f_k}(n; M, A) = 
\frac{1}{\phi(M)}\sum_{\chi \bmod M}\sum_{a\in A} \bar\chi(a) \pi_{f_k}(n,\chi).
\end{equation*}
Hence for any $A, B \subset (\mathbf{F}_{q}[t]/(M))^*$, one has
\begin{align*}
    \frac{1}{\lvert A\rvert}\pi_{f_k}(n; M, A) - \frac{1}{\lvert B\rvert}\pi_{f_k}(n; M, B)
    &=  \frac{1}{\phi(M)}\sum_{\chi \bmod M}\(\frac{1}{\lvert A\rvert}\sum_{a\in A} \bar\chi(a) - \frac{1}{\lvert B\rvert}\sum_{b\in B} \bar\chi(b) \) \pi_{f_k}(n,\chi) \\
    &= \sum_{\chi \bmod M}c(\chi,A,B)  \pi_{f_k}(n,\chi). 
\end{align*}
Note that  the case $\chi = \chi_0$ is trivial, one has $c(\chi,A,B) =0$. 

We have $\vert c(\chi,A,B) \rvert \leq 2$, so when we sum over the degree $n$, the implicit constants in the error terms are at most multiplied by $2$.

Now, let us sum over the degree. We divide the range $n\leq X$ into the two parts $n\leq \frac{X}{3}$ and $\frac{X}{3}<n\leq X$. For $n\leq \frac{X}{3}$, we use the trivial bound $\pi_{f_k}(n; M, A)\leq q^{n}$. We have 
\begin{align*}
\Delta_{f_k}(X; M, A, B)&= \sum_{\chi \bmod M}c(\chi,A,B)  \frac{X (k-1)!}{q^{X/2}(\log X)^{k-1}}\Bigg\{\sum_{n\leq \frac{X}{3}}+\sum_{\frac{X}{3}<n\leq X}\Bigg\}   \pi_{f_k}(n,\chi)\nonumber\\
&= \sum_{\chi \bmod M}c(\chi,A,B)\frac{X (k-1)!}{q^{X/2}(\log X)^{k-1}}\sum_{\frac{X}{3}<n\leq X} \pi_{f_k}(n,\chi) +O\(Xq^{1 - \frac{X}{6}}\frac{(k-1)!}{(\log X)^{k-1}}\).
\end{align*}
When $\frac{X}{3}<n\leq X$, we have $k=o((\log X)^{\frac12})=o((\log n)^{\frac12})$, the asymptotic formula in Theorem~\ref{Prop_k_general} yields
\begin{align*}
&\sum_{\frac{X}{3}<n\leq X}  \pi_{f_k}(n,\chi)\nonumber\\
&= \frac{(-1)^k}{(k-1)!}  \sum_{\frac{X}{3}<n\leq X} \Bigg\{\frac{q^{n/2}(\log n)^{k-1}}{n}
 \( \(m_+(\chi) -\epsilon_f\frac{\delta(\chi^2)}{2}\)^k+(-1)^n \(m_-(\chi) -\epsilon_f\frac{\delta(\chi^2)}{2}\)^k \) \\
 &\quad + \sum_{\alpha_j\neq\pm\sqrt{q}} m_j(\chi)^k \frac{\alpha_j^n(\chi) (\log n)^{k-1}}{n}
+O \left( d^k k(k-1)\frac{1}{\gamma(\chi)} \frac{q^{n/2}(\log n)^{k-2}}{n}  + d\frac{q^{n/3}}{n}\right) \Bigg\}\\
&= \frac{(-1)^k}{(k-1)!}  \sum_{n\leq X} \Bigg\{\frac{q^{n/2}(\log n)^{k-1}}{n}
 \( \(m_+(\chi) -\epsilon_f\frac{\delta(\chi^2)}{2}\)^k+(-1)^n \(m_-(\chi) -\epsilon_f\frac{\delta(\chi^2)}{2}\)^k \) \\
 &\quad + \sum_{\alpha_j\neq\pm\sqrt{q}} m_j(\chi)^k \frac{\alpha_j^n(\chi) (\log n)^{k-1}}{n}
+O \left( d^k \frac{k(k-1)}{(k-1)!\gamma(\chi)} \frac{q^{n/2}(\log n)^{k-2}}{n} + d\frac{q^{n/3}}{n}  \right) \Bigg\} \\
&\qquad+  O \left( d^k \frac{1}{(k-1)!} \frac{q^{X/6}(\log X)^{k-1}}{X}  \right).
\end{align*}
Now, applying Lemma~\ref{Lem_SumOverN} for each $\alpha_j = \sqrt{q}e^{i\gamma_j(\chi)}$ (real or not), and using $k= o(\log X)$,
    one has
    \begin{align*}
\frac{X}{(\log X)^{k-1}q^{X/2}}\sum_{n=1}^{X}\frac{\alpha_j^n(\log n)^{k-1}}{n}  = \frac{\alpha_j}{\alpha_j -1}\(\frac{\alpha_j}{\sqrt{q}}\)^{X} + O\( \frac{1}{X\log X}\).
\end{align*}
We also apply Lemma~\ref{Lem_SumOverN} and \cite[Lem. 2.2]{Cha2008} to the sum of the error term and derive that
\begin{align*}
&\Delta_{f_k}(X; M, A, B)\nonumber\\
&= (-1)^k \sum_{\chi}c(\chi,A,B)\Bigg(\ \(m_+(\chi) +\frac{\delta(\chi^2)}{2}\)^k \frac{\sqrt{q}}{\sqrt{q}-1} +(-1)^X \(m_-(\chi)+\frac{\delta(\chi^2)}{2}\)^k \frac{\sqrt{q}}{\sqrt{q}+1} \nonumber\\
&\qquad +\sideset{}{'}\sum_{j=1}^{d_{\chi}} m_j^k e^{iX\gamma_{j}(\chi)} \frac{\alpha_j(\chi)}{\alpha_j(\chi)-1} \Bigg)  + O\( \frac{d^k k(k-1)}{\gamma(M)\log X} + dq^{-X/6}\).
\end{align*}
This concludes the proof of Theorem~\ref{Th_Difference_k_general_deg<X}.
\end{proof}

\section*{Acknowledgements}
The authors thank Peter Humphries and Lior Bary-Soroker for suggesting the project.
 The authors are grateful to Daniel Fiorilli for his feedback and careful reading. The authors are also grateful to Byungchul Cha and Andrew Granville for their helpful advice and for pointing out the works of Wanlin Li and Sam Porritt respectively. We are grateful for Wanlin Li's explanation of her paper and her help in finding interesting examples, and to Sam Porritt for sending us his preprint.
 This paper also benefited from conversations with Florent Jouve, Jon Keating, Corentin Perret-Gentil, and K. Soundararajan.
 We would like to thank the CRM, McGill University, Concordia University, the University of Ottawa and MPIM for providing good working conditions that made this collaboration possible. 

The computations in this paper were performed using SageMath and Matlab.

\bibliographystyle{amsalpha} 
\bibliography{biblio}

\end{document}